\documentclass[letterpaper, 10pt, dvipsnames]{article}

\usepackage{type1ec}
\usepackage[T1]{fontenc}
\usepackage[utf8]{inputenc}    
\usepackage[english]{babel}
\usepackage[normalem]{ulem}
\usepackage{dsfont, bm, pifont, mathrsfs}
\usepackage{stmaryrd}
\usepackage{bbm}
\usepackage{enumitem}
\usepackage{bold-extra}
\usepackage{float, graphicx, wrapfig, tikz}
\usetikzlibrary{positioning, fit}
\usepackage[margin=0.8in]{subcaption}
\usepackage[font=small, margin=0.6in]{caption}

\usepackage{amsthm, amsmath, amsfonts, amssymb, mathrsfs, mathtools}
\usepackage[alphabetic]{amsrefs}
\usepackage{hyperref, xcolor, titlesec} 
\hypersetup{colorlinks = true, urlcolor = RoyalBlue!70!Black,linkcolor=BrickRed, citecolor=Green!80!Yellow, bookmarksopen = true}
\usepackage[nomarginpar]{geometry}
\numberwithin{equation}{section}
\usepackage[boxsize = .3em]{ytableau}
\usepackage{cleveref}
\DeclareMathOperator{\Id}{Id}

\DeclareMathOperator{\Var}{Var}
\renewcommand{\Im}{\operatorname{Im}}
\renewcommand{\Re}{\operatorname{Re}}
\DeclareMathOperator{\Tr}{Tr}
\DeclareMathOperator{\tr}{tr}

\newcommand{\e}{\mathrm{e}}
\renewcommand{\i}{\mathrm{i}}
\let\d\relax
\newcommand{\d}{\mathrm{d}}

\newcommand{\E}{\mathds{E}}

\renewcommand{\O}{\mathbf{O}}

\newcommand{\Cbb}{\mathbb{C}}

\newcommand{\Rbb}{\mathbb{R}}

\newcommand{\Bcal}{\mathcal{B}}
\newcommand{\Ccal}{\mathcal{C}}
\newcommand{\Dcal}{\mathcal{D}}

\newcommand{\Fcal}{\mathcal{F}}

\newcommand{\Hcal}{\mathcal{H}}
\newcommand{\Ical}{\mathcal{I}}

\newcommand{\Mcal}{\mathcal{M}}

\newcommand{\Pcal}{\mathcal{P}}
\newcommand{\Qcal}{\mathcal{Q}}
\newcommand{\Rcal}{\mathcal{R}}
\newcommand{\Scal}{\mathcal{S}}

\newcommand{\Ycal}{\mathcal{Y}}

\newcommand{\Escr}{\mathscr{E}}

\newcommand{\Efr}{\mathfrak{E}}

\renewcommand{\epsilon}{\varepsilon}
\renewcommand{\phi}{\varphi}

\let\O\relax
\newcommand{\O}[1]{\mathcal{O}\left(#1\right)}

\newcommand{\Pds}{\mathds{P}}
\newcommand{\Eds}{\mathds{E}}
\newcommand{\cp}{\color{magenta}}

\newcommand{\eer}{\normalcolor}

\newcommand{\unn}[2]{\left[\!\left[#1,#2\right]\!\right]}

\newtheorem{ccounter}{ccounter}[section]
\newtheorem{theorem}[ccounter]{Theorem}
\newtheorem{lemma}[ccounter]{Lemma}
\newtheorem{corollary}[ccounter]{Corollary}
\newtheorem{definition}[ccounter]{Definition}
\newtheorem{proposition}[ccounter]{Proposition}

\theoremstyle{definition}
\newtheorem{remark}[ccounter]{Remark}

\titleformat{\section}[block]{\normalfont\filcenter}{\large\Roman{section}.}{.7em}{\large\scshape}
\newcommand{\appendixswitch}{\titleformat{\section}[block]{\normalfont\filcenter}{\large\Alph{section}.}{.7em}{\large\scshape}}
\titleformat{\subsection}[runin]{\normalfont}{\it\bf \thesubsection .}{.5em}{\it}[.]
\titleformat{\subsubsection}[runin]{\normalfont}{\bf \thesubsubsection .}{.5em}{\bf}[.]
\titleformat*{\paragraph}{\itshape\mdseries}
\renewcommand{\abstractname}{\normalsize\normalfont\scshape Abstract}

\begin{document}
\tikzset{every node/.style={circle, minimum size=.1cm, inner sep = 2pt, scale=1}}

%Weak

\title{\vspace{-5ex}\bfseries\scshape{Local law and delocalization for the Sachdev--Ye--Kitaev model}}
\author{L. \textsc{Benigni}
\\\vspace{-0.15cm}\footnotesize{\it{Universit\'e de Montr\'eal}}\\\footnotesize{\it{lucas.benigni@umontreal.ca}} \and G. \textsc{Cipolloni}\\\vspace{-0.15cm}\footnotesize{\it{ Universit\`a di Roma Tor Vergata}}\\\footnotesize{\it{cipolloni@axp.mat.uniroma2.it}}}
\date{}
\maketitle
\renewcommand{\abstractname}{\normalsize\normalfont\scshape Abstract}

\begin{abstract}
    We establish a mesoscopic local law and quantitative eigenvector-delocalization estimates for the Sachdev--Ye--Kitaev model (SYK) Hamiltonian of $N$ interacting Majorana fermions subject to a $q$-body interaction. For even $q\ll N^{\frac{1}{2}}$, we prove that the normalized Stieltjes transform is approximated, uniformly on bounded energy intervals, by an explicit deterministic profile which we construct as an asymptotic expansion in the proportion $\alpha_{N,q}$ of anticommuting $q$-body interaction terms. 
    The result holds both on the full Hilbert space and in each fermion-parity sector. We derive mesoscopic eigenvalue counting and spectral form factor estimates. For fixed $q$, we further obtain high-probability $\ell^\infty$-delocalization bounds in any deterministic orthonormal basis for individual bulk eigenvectors. These are the first mesoscopic laws and eigenvector delocalization estimates for the SYK Hamiltonian, which we obtain using the resolvent method.
\end{abstract}
\section{Introduction and main results}
The Sachdev--Ye--Kitaev model originated from a random spin system introduced  in \cite{sachdevye}. Their starting point was a Heisenberg quantum magnet built from a collection of $N$ quantum spins and an infinite-range Hamiltonian of the form 
\[
H_{\mathrm{SY}} = \frac{1}{\sqrt{N\Mcal}}\sum_{1\leqslant i<j\leqslant N}J_{ij}S_i\cdot S_j,
\]
where the interaction coefficients $J_{ij}$ are independent centered Gaussian variables with fixed variance. The main novelty of \cite{sachdevye} was to consider spins from a high-dimensional space $\mathrm{SU}(\Mcal)$, the case of $\mathrm{SU}(2)$ was first considered in \cite{brai1980replica}, and to consider the limit where $N$ and then $\Mcal$ becomes large.
%We write $\Mcal$ here rather than $M$, the latter being reserved from \eqref{eq:sykhamil} on for the number $\binom Nq$ of interaction terms of $H_q$.
It was then observed that, in this limit, the expected correlation functions are described by deterministic nonlinear integral equations which have solutions whose correlations decay as a power-law instead of exponentially, also called a \emph{gapless spin-fluid phase} in the physics literature. 

Kitaev in \cites{Kitaev1, Kitaev2} introduced a simpler model in which the spin variables are replaced by Majorana fermions thus now depending on a single dimension. For $N$ even, let $\gamma_1,\dots,\gamma_N$ be Majorana fermions satisfying the relation 
$
\gamma_i\gamma_j+\gamma_j\gamma_i = 2\delta_{ij}\Id.
$
We then define the model as, with $M=\binom{N}{q}$, 
\begin{equation}
H_q = \frac{\i^{\binom{q}{2}}}{\sqrt{M}}\sum_{i_1<\dots<i_q}J_{i_1,\dots,i_q}\gamma_{i_1}\dots\gamma_{i_q}
\quad\text{with}\quad 
J_{i_1,\dots,i_q}\sim \mathcal{N}(0,1)\text{ i.i.d.}
\end{equation}
This is what is now known as the Sachdev--Ye--Kitaev (SYK) model. Even though this model and the one introduced by Sachdev and Ye are different random operators on different Hilbert spaces, their connections come from the fact that their two-point functions satisfy closely related nonlinear self-consistency equations. Namely, define for $\tau\in (0,\beta)$, the averaged two-point function
\[
G_N(\tau) = \frac{1}{N}\sum_{i=1}^N \Eds\left[  
    \frac{\Tr\left(\e^{-\beta H_q}\gamma_i(\tau)\gamma_i(0)\right)}{\Tr(\e^{-\beta H_q})}
\right]
\quad\text{where}\quad 
\gamma_i(\tau) = \e^{\tau H_q}\gamma_i\e^{-\tau H_q}.
\]
We can extend $G_N$ to $\Rbb$ by fermionic antiperiodicity i.e. $G_N(\tau+\beta)=-G_N(\tau)$. Here, and only in this paragraph, $\tau$ denotes the imaginary time; from Section~\ref{sec:llaw} on, and throughout the statements of our results, $\tau\in\{\text{Full},+,-\}$ labels the fermion-parity sector. Then for large $N$, it was observed in \cites{maldacena2016remarks, polchinski2016spectrum} that $G_N$ should converge to a deterministic function $G$ characterized by the nonlinear  distributional equation for $\tau\in(0,\beta)$
\[
\partial_\tau G(\tau) - \int_0^\beta G(\tau-s)^{q-1}G(s)\d s = \delta(\tau).
\]
In particular, in the low-temperature regime $\beta \gg 1$ and for $\vert \tau\vert \ll \beta$, we obtain the approximate solution 
\[
{G}_{\mathrm{LT}}(\tau) \asymp \mathrm{sign}(\tau)\pi^{\frac{2}{q}}\left({\beta\sin\left(\pi\frac{\vert \tau\vert}{\beta}\right)}\right)^{-\frac{2}{q}}
\]
showing a power-law singularity. This has been shown rigorously in \cite{he2026spectral}.

In the mathematical treatment of the model, one of the first results was the convergence of the global law for the SYK model in \cite{feng2019spectrum} using the moment method. They show that as long as $q\ll \sqrt{N}$ the limiting eigenvalue distribution is given by the standard Gaussian law. Namely, if one denotes by $\lambda_1\leqslant \dots\leqslant \lambda_D$ the eigenvalues of $H_q$\footnote{Note that $\lambda_i\in\Rbb$ since $H$ is Hermitian through the anticommutation relation.} then we have that
    \begin{equation}\label{eq:global}
    \mu_N \coloneqq \frac{1}{2^{\frac{N}{2}}}\sum_{i=1}^{2^{\frac{N}{2}}} \delta_{\lambda_i}
    \xRightarrow[N\to\infty]{}\mu \coloneqq \frac{1}{\sqrt{2\pi}}\e^{-\frac{x^2}{2}}\d x
    \quad a.s.
    \end{equation}
These results were strengthened in follow-up works by the same authors in \cites{feng2020spectrum, feng2021spectrum}, where the authors proved a central limit theorem for linear statistics as well as a result on concentration of measures. If $q\asymp \sqrt{N}$, the limiting global law is given by another distribution linked to the $q$-Hermite polynomials, while if $q\gg \sqrt{N}$, the limiting law is given by the semicircle distribution as for Wigner matrices. This phenomenon for the limiting density of states was first observed for mixed $p$-spin models in \cites{erdos2014phase, keating2015}, using similar methods. In fact, for any $q\ll\sqrt{N}$, the density of states at finite $N$ is described more precisely by the \emph{$Q$-Gaussian} law $\nu_Q$ of Bo$\dot{z}$ejko and Speicher \cites{bozejko91, bozejko97}, with $Q=1-2\alpha_{N,q}$ and $\alpha_{N,q}$ defined in \eqref{eq:alphanqasymp} below. 

On the global scale, such corrections to the Gaussian profile were computed by means of the moment method in the physics literature, see e.g. \cite{garciagarcia2016}. More precisely, the $Q$-Hermite approximation to the finite-$N$ density of states was obtained in \cite{garciagarcia2017} by resumming the crossing contributions to the moments, the moments of $H_q$ were then computed exactly up to order $N^{-2}$ in \cite{garciagarcia2018}, and up to order $N^{-3}$ in \cite{jia2018}, by enumerating chord-intersection graphs; see also \cite{jia2020}, where the fluctuations of the $Q$-Hermite coefficients are used to study the number variance and the spectral form factor. In particular, it is already visible from these expansions that the density of states of $H_q$ and the $Q$-Gaussian law disagree at order $N^{-2}$.

One outcome of Theorem~\ref{theo:locallaw} below is that this description of the density of states persists all the way down to mesoscopic spectral scales, and that it can be pushed to an arbitrary order in $\alpha_{N,q}$. To the best of our knowledge, the full expansion has not appeared before, even on global scales: the mechanism which produces it is entirely different from, and considerably lighter than, the enumeration of intersection graphs, since each order is obtained from the previous ones by a single explicit integration of a first order linear differential equation, see Definition~\ref{def:thetas} below. We stress that the double-scaled regime $q\asymp\sqrt{N}$, in which a far more detailed and genuinely nonperturbative picture is available through chord diagrams and the associated transfer-matrix and quantum-group structures \cites{berkooz2019, berkooz2020susy, berkooz2021complex}, is not covered by our results; we refer the reader to the review \cite{berkoozmamroud2025} and to the references therein. \eer It is important to note that all these results concern the global properties of the spectrum. 

The spectral edge of SYK has seen substantial study recently, by means of quite different techniques. The SYK model exhibits a striking separation between its bulk and edge scales. While we saw that the limiting eigenvalue distribution is the Gaussian law, and most eigenvalues lie on an order-one scale, the largest eigenvalue is of order $\mathcal{O}(\sqrt{N})$. This behavior reflects the highly correlated structure of the Hamiltonian: it is a Gaussian linear combination of $\sim N^q$ terms acting on a Hilbert space of dimension $2^{\frac{N}{2}}$. The spectral edge thus occurs on a scale independent of the global Gaussian limit, a behavior that differs sharply from classical random matrix ensembles where typically the limiting density has compact support and the largest eigenvalue remains on the same scale. The work \cite{feng2019spectrum} gives an upper bound on the spectral edge which was sharpened in terms of constant in \cite{hastings2022}. A first lower bound was obtained in the physics literature in \cite{anschuetz2025}. Very recently, \cite{he2026spectral} gives the almost sure convergence of $N^{-\frac{1}{2}}\lambda_D$ for the
SYK model; see also \cite{basu2026sharp} for sharp bounds on the spectral edge of the SYK model (and a sparse version of it) for any $q$. 
For the free energy, the annealed limit was recently computed in \cites{gamarnik2026free}.

While the global distribution of eigenvalues and the spectral edge are now somewhat understood, the microscopic and mesoscopic theory beyond the edge is open and the goal of this work is to push further in this direction. For local statistics, considering the quartic SYK for simplicity (see \cites{Kanazawa2017, sun2020periodic} for an analogous prediction for general $q$), a fascinating behavior is suggested by the physics literature in the works \cites{you2017sachdev, Kanazawa2017}: for a bulk energy $E$, after  selecting a parity sector and removing some symmetry-enforced multiplicity of eigenvalues, then along a sequence with fixed $N\mod 8$,
\[
\sum_{j=1}\delta_{\frac{\lambda_j-E}{\Delta_N(E)}}
\Longrightarrow\mathrm{Sine}_\beta 
\quad\text{where}\quad 
\beta = \left\{ 
    \begin{array}{ll}
    1 &\text{if }N = 0\mod 8,\\
    2 &\text{if }N = 2,6 \mod 8,\\
    4 &\text{if }N = 4 \mod 8,
    \end{array}
    \right.
\]
and $\Delta_N(E) = 2^{-\frac{N}{2}}(\rho_N(E))^{-1}$ with $\rho_N(E)$ the density of states. For control of eigenvectors, the delocalization and eigenstate thermalization hypothesis is suggested in the physics literature in \cites{Sonner2017, haque2019}, where they also study the spectral form factor of the SYK model.  However, to the best of our knowledge, there were no mathematical results on eigenvector delocalization or local law for this model previous to this work. In this article we aim to give the first results in this direction by understanding the convergence \eqref{eq:global} from \cite{feng2019spectrum} at a smaller spectral scale. In particular, we give an alternative proof of the convergence in \cite{feng2019spectrum} which does not rely on the moment method. As a corollary of this result we also present the asymptotic for fairly short, but still $N$-dependent, times of the spectral form factor.  Moreover, we give the first rigorous delocalization bounds for bulk eigenvectors. Our main model of interest is the SYK Hamiltonian $H_q$ for fixed $q$; however, for most results, our approach works for any $q\ll \sqrt{N}$, we thus consider this case directly. The case $q\gtrsim\sqrt{N}$ likely requires a different approach, we will develop this in future works. 

Our proof of the local law is based on a resolvent approach adapted to the Clifford structure of the SYK Hamiltonian.  For the averaged local law, Gaussian integration by parts gives an approximate differential equation for the expected Stieltjes transform, of the form
\[
f'(z)+zf(z)+1=\sum_{r\geqslant 1}\frac{c_r\alpha_{N,q}^r}{(2r+1)!}f^{(2r+1)}(z),
\]
where $\alpha_{N,q}$ denotes the proportion of anticommuting interaction terms (see \eqref{eq:alphanqasymp} below) and the constants $c_r$ are bounded. Its leading part, obtained by discarding all the terms proportional to $\alpha_{N,q}$, is solved by the Stieltjes transform $m$ of the standard Gaussian law. The whole equation can then be solved perturbatively in $\alpha_{N,q}$ within the linear span of the derivatives of $m$: the reason is that the differential operator $f\mapsto f'+zf$ maps $m^{(k)}$ to a multiple of $m^{(k-1)}$, so that each order of the expansion is obtained from the previous ones by a single explicit integration. Truncating the resulting expansion after $R$ terms produces a deterministic profile which solves the equation up to an error of order $\alpha_{N,q}^{R+1}$; for $R=1$ this profile coincides, up to an error of order $\alpha_{N,q}^{2}$, with the Stieltjes transform of the $Q$-Gaussian measure of \cites{bozejko91,bozejko97} with $Q=1-2\alpha_{N,q}$, but for $R\geqslant2$ it does not. 

A stability argument  for the resulting differential equation, and Gaussian concentration then yield the local law.  We stress that unlike in the Wigner case (and its generalization), for the SYK model the expected Stieltjes transform is an approximate solution of an ordinary differential equation rather than a polynomial equation.  The short-time control of the spectral-form-factor then follows by Fourier analysis, as a corollary of the averaged local law. For the isotropic local law and thus delocalization estimates on eigenvectors, the general approach is similar, but it requires important new inputs. In particular, one needs to control objects of the form
\[
\langle {\bm x},G\gamma_{i_1}\cdots \gamma_{i_q} G{\bm y}\rangle
\]
in order to apply Gaussian concentration as before. Here ${\bm x}, {\bm y}$ are deterministic unit vectors and $G=G(z):=(H_q-z)^{-1}$ is the resolvent. The main difficulty is thus to sum the square of this quantity over all $q$-body Clifford monomials without losing the exponentially large Hilbert-space dimension. For this purpose we rely on state-of-the-art fermionic operator estimates \cites{christiansen2024hilbert, visconti2026}.  Thus the two main novelties are the use of the sparse anticommutation geometry of the Clifford family as the small parameter in the averaged local law, and the fermionic many-body estimate that makes isotropic concentration possible despite the exponentially large matrix dimension. 
While in this work we specifically focus on the SYK model, we expect that the approach used in this work has the potential to give analogous results for other many-body Hamiltonians.

Although the scales obtained here are substantially weaker than the optimal local laws available for classical random-matrix ensembles, they should be viewed in a different context. The SYK Hamiltonian acts on a Hilbert space of dimension $2^{\frac{N}{2}}$, while being generated by only polynomially many ($\sim N^q$) random variables through strongly correlated, noncommuting many-body interactions. It also exhibits features absent from standard one-particle models, including nontrivial symmetry sectors and a separation between an order-one Gaussian bulk and a spectral edge of order $\sqrt N$. The extensive theory of strong local laws and eigenvector universality for classical random matrices does not directly apply in this setting, and rigorous local laws for genuinely interacting many-body Hamiltonians remain comparatively scarce. The present results therefore provide a first step toward a mathematical understanding of mesoscopic spectral statistics and eigenstate thermalization in the SYK model.
\subsection{Definition of the model}
Let $N$ be even and $\gamma_1,\dots,\gamma_N$ be Majorana Fermions satisfying the relation
\[
\gamma_i\gamma_j+\gamma_j\gamma_i = 2\delta_{ij}\Id_D
\quad\text{with}\quad 
D = 2^{n}\quad\text{and}\quad n=\frac{N}{2}.
\] For an explicit representation of the Majorana fermions, we can consider the Hilbert space 
\[
\Hcal_N = (\Cbb^2)^{\otimes n}
\quad\text{such that }\quad 
\dim(\Hcal_N) = D.
\]
From the Pauli matrices 
\[
\sigma_x = \begin{bmatrix} 
0 & 1\\1 & 0
\end{bmatrix},
\quad 
\sigma_y = \begin{bmatrix}
    0 & -\i\\ \i & 0
\end{bmatrix},
\quad\text{and}
\quad
\sigma_z = \begin{bmatrix}
    1 & 0 \\ 0 & -1
\end{bmatrix},
\]
we can use the Jordan--Wigner representation of the Majorana Fermions 
\[
\gamma_{2j-1} = \sigma_z ^{\otimes (j-1)}\otimes \sigma_x\otimes \Id_2^{\otimes (n-j)}
\quad\text{and}\quad 
\gamma_{2j} = \sigma_z^{\otimes (j-1)}\otimes \sigma_y\otimes \Id_2^{\otimes (n-j)}
\quad\text{for}\quad 1\leqslant j\leqslant n.
\]
In particular, we have that the corresponding Clifford algebra is isomorphic to the space of matrices of size $2^{n}$ 
\[
\mathrm{Cl}_{N}(\Cbb) \coloneqq 
\langle e_1,\dots,e_N:e_ie_j+e_je_i=2\delta_{ij}\rangle \cong \Cbb^{D\times D}.
\]
For a given $q$-subset $A=\{i_1<\dots<i_q\}\subset \unn{1}{N}$ we define 
\[
\Gamma_A = \Gamma_{i_1,\dots,i_q} \coloneqq \i^{\binom{q}{2}} \gamma_{i_1}\dots\gamma_{i_q}.
\]
We thus define our model as, with $M=\binom{N}{q}$,
\begin{equation}
\label{eq:sykhamil}
H_q = \frac{1}{\sqrt{M}}\sum_{i_1<\dots<i_q}J_{i_1,\dots,i_q}\Gamma_{i_1,\dots i_q}
=
\frac{1}{\sqrt{M}}\sum_{\substack{A\subset\unn{1}{N}\\\vert A\vert = q}}J_A\Gamma_A
\quad\text{with}\quad 
J_A = J_{i_1,\dots,i_q}\sim \mathcal{N}(0,1)\text{ i.i.d.}
\end{equation}
Define the fermion-parity operator
\[
P_N = (-\i)^n\gamma_1\dots\gamma_{N} = \sigma_z^{\otimes n}.
\]
Notice that it is a self-adjoint involution, since 
\[
P_N^\ast = P_N,\quad P_N^2 = \Id_D, 
\quad\text{and also} \quad P_N\gamma_j = -\gamma_jP_N.
\]
This gives the natural decomposition of the Hilbert space $\Hcal_N$:
\[
\Hcal_N = \Hcal_N^+ \oplus\Hcal_N^-
\quad\text{with}\quad 
\Hcal_N^{\pm} = \ker(P_N \mp\Id_D)
\quad\text{and}\quad 
D^{\pm} \coloneqq \dim(\Hcal_N^{\pm}) = \frac{D}{2}.
\]
Since we have the identity 
\[
P_N\Gamma_A = (-1)^{\vert A\vert}\Gamma_AP_N,
\]
we see that if $q$ is even, every $\Gamma_A$ preserves the parity sectors and we can write $H_q$ in the block form 
\[
H_{q} = \begin{bmatrix}
    H_q^+ & 0\\ 0 & H_q^-
\end{bmatrix}
\quad\text{on}\quad 
\Hcal_N^+\oplus \Hcal^-_N.
\]
We can define $\Tr$ the usual unnormalized matrix trace on $\Hcal_N$ and $\Tr^{\pm}$ the one on $\Hcal_N^{\pm}$. We can then define for an operator $X$ commuting with $P_N$, 
\[
\tr(X) \coloneqq \frac{1}{D}\Tr(X),
\quad 
\tr_{\pm}(X) \coloneqq \frac{2}{D}\Tr(\Pi_N^{\pm}X)
\quad\text{where}\quad 
\Pi_N^{\pm}\text{ are the orthogonal projections }
\Pi_N^{\pm} = \frac{1}{2}(\Id_D\pm P_N).
\]
In particular we have the identities 
\[
\tr_{\pm}(X) = \tr(X)\pm \tr(P_NX)
\quad\text{and}\quad 
\tr(X) = \frac{1}{2}(\tr_+(X)+\tr_-(X)).
\]We denote for $z\in \Cbb_+$ and $\tau\in\{\text{Full},+,-\}$
\[
G_\tau(z) = (H_q^{\tau}-z\Id_{D^\tau})^{-1},\quad s_{N,\tau}(z) =\tr_\tau(G_\tau(z))
\quad\text{and} \quad 
m_{N,\tau}(z) = \Eds\left[s_{N,\tau}(z)\right].
\]
For any subset $A\subset\unn{1}{N}$ with $\vert A\vert =q$, we also define the parameter
\begin{equation}
\label{eq:alphanqasymp}
\alpha_{N,q} = \frac{1}{M}\left\vert 
    \left\{
        B\subset\unn{1}{N},\,\vert B\vert =q,\,\Gamma_A\Gamma_B = -\Gamma_B\Gamma_A  
    \right\}    
\right\vert
=
\frac{1}{M}\sum_{\substack{0\leqslant r\leqslant q\\ r\text{ odd}}} 
\binom{q}{r}\binom{N-q}{q-r} = \frac{q^2}{N}+\O{\frac{q^4}{N^2}},
\end{equation}
where the asymptotics holds when $q\ll \sqrt{N}$. Note that $\alpha_{N,q}>0$ whenever $2\leqslant q$ and $N\geqslant 2q-1$, since the term $r=1$ of the sum is already positive; this will matter in Definition~\ref{def:thetas}, where the constants $c_r$ are defined by dividing by $\alpha_{N,q}^r$. Note that from the convergence \eqref{eq:global} of \cite{feng2019spectrum}, giving the asymptotic eigenvalue density of $H_q$ through the moment method,  we can obtain pointwise convergence of the Stieltjes transform: for $z\in \Cbb_+$
    \[
    s_N(z) \xrightarrow[N\to\infty]{a.s} m(z)= \int_\Rbb \frac{\e^{-\frac{x^2}{2}}\d x}{\sqrt{2\pi}(x-z)}
    \quad\text{which satisfies }
    m'(z)+zm(z)+1=0
    \quad\text{with}\quad m(z)\sim-\frac{1}{z} \text{ as }\vert z\vert\to\infty.
    \]

Our first result is a local law strengthening the convergence above to finer spectral scales. On local scales, the deterministic approximation of $s_N$ is still given at leading order by $m(z)$, but we also identify \emph{all} the sub-leading corrections in $\alpha_{N,q}$, which is what allows us to reach an error which is an arbitrarily large power of $N^{-1}$ for fixed $q$. Before stating the result we must therefore construct these corrections.

The construction rests on the elementary observation, proved in Lemma~\ref{lem:gaussprofile} below, that the first order differential operator
\begin{equation}
\label{eq:defL}
L[f](z)\coloneqq f'(z)+zf(z)
\end{equation}
maps each derivative of $m$ to a multiple of the previous one, namely $L[m^{(k)}]=-k\,m^{(k-1)}$ for $k\geqslant1$, while $L[m]=-1$. Consequently, every inhomogeneous equation $L[h]=m^{(j)}$ is solved explicitly by $h=-\frac{1}{j+1}m^{(j+1)}$, and this is its only solution which vanishes as $\Im z\to\infty$; indeed, the solutions of the homogeneous equation $L[h]=0$ are the multiples of $\e^{-\frac{z^2}{2}}$, which are unbounded along vertical lines unless they vanish identically. This is what generates the following hierarchy.

\begin{definition}
\label{def:thetas}
    Let $\mu_{2k}=\mu_{2k}(N,q)\coloneqq \Eds\left[\tr\left(H_q^{2k}\right)\right]$ be the even moments of the density of states of $H_q$, the odd ones vanishing since $H_q$ and $-H_q$ have the same law, so that $\mu_0=\mu_2=1$ and $\mu_4=3-2\alpha_{N,q}$ by \eqref{eq:fourthmomentsyk} below. Define recursively the constants $c_r=c_r(N,q)$, for $r\geqslant1$, by
    \begin{equation}
    \label{eq:crdef}
    c_r\,\alpha_{N,q}^r\coloneqq (2r+1)\mu_{2r}-\mu_{2r+2}-\sum_{s=1}^{r-1}\binom{2r+1}{2s+1}c_s\,\alpha_{N,q}^s\mu_{2r-2s}.
    \end{equation}
    Set then $\Theta_0\coloneqq m$ and, recursively for $s\geqslant1$, let $\Theta_s$ be the unique solution vanishing as $\Im z\to\infty$ of
    \begin{equation}
    \label{eq:thetarec}
    L\left[\Theta_s\right](z)=\sum_{r=1}^{s}\frac{c_r}{(2r+1)!}\,\Theta_{s-r}^{(2r+1)}(z).
    \end{equation}
    Finally, for an integer $R\geqslant0$ we set
    \begin{equation}
    \label{eq:defphiR}
    \Phi_R\coloneqq \sum_{s=0}^{R}\alpha_{N,q}^s\,\Theta_s.
    \end{equation}
\end{definition}

\begin{remark}
Denote by $\kappa_j$ the cumulants of $\E \mu_N$. Then, by the standard cumulant relation formula
\[
\mu_{2r+2}=(2r+1)\mu_{2r}+\sum_{s=1}^r\binom{2r+1}{2s+1}\kappa_{2s+2}\mu_{2r-2s},
\]
we readily see that $c_r\alpha_{N,q}^r=-\kappa_{2r+2}$.
\end{remark}

Although $\mu_{2k}$ is defined through the trace on the full space, the same quantity is obtained in either parity sector as soon as $2kq<N$, so that Definition~\ref{def:thetas} produces one and the same profile $\Phi_R$ for the three sectors $\tau\in\{\text{Full},+,-\}$. 

Two comments are in order. First, the recursion \eqref{eq:thetarec} is nothing but the requirement that $\Phi_R$ solve, up to an error of order $\alpha_{N,q}^{R+1}$, the differential equation
\begin{equation}
\label{eq:heuristicode}
L[f](z)+1=\sum_{r\geqslant1}\frac{c_r\alpha_{N,q}^r}{(2r+1)!}f^{(2r+1)}(z)
\end{equation}
which is satisfied approximately by $m_{N,\tau}$, as we shall see in Proposition~\ref{pro:highode}; one simply collects the terms of equal order in $\alpha_{N,q}$ on both sides.

The structure of the profiles $\Theta_s$ is described by the following lemma, whose proof is postponed to Section~\ref{sec:llaw}.
\begin{lemma}
\label{lem:thetastruct}
    For every $s\geqslant1$ the function $\Theta_s$ is well defined and
    \[
    \Theta_s\in\mathrm{span}\left\{m^{(2s+2)},m^{(2s+4)},\dots,m^{(4s)}\right\},
    \]
     with coefficients which are universal polynomials in $c_1,\dots,c_s$. In particular $\Theta_s=\mathcal{D}_sm$ for a constant-coefficient differential operator $\mathcal{D}_s$ involving only even derivatives of order at most $4s$. 
    Moreover, for all integers $s\geqslant1$ and $l\geqslant0$ there is $K_{s,l}>0$, depending also on $\max_{r\leqslant s}\vert c_r\vert$, such that
    \begin{equation}
    \label{eq:thetabounds}
    \left\vert \Theta_s^{(l)}(E+\i\eta)\right\vert \leqslant K_{s,l}
    \quad\text{for }E\in\Rbb,\ 0<\eta\leqslant2,
    \qquad\text{and}\qquad
    \left\vert \Theta_s^{(l)}(E+\i\eta)\right\vert \leqslant \frac{K_{s,l}}{\eta^{2s+3+l}}
    \quad\text{for }\eta\geqslant1.
    \end{equation}
    Finally, $\Phi_R$ is the Stieltjes transform of the signed density
    \begin{equation}
    \label{eq:defvarrhoR}
    \varrho_R\coloneqq \rho_G+\sum_{s=1}^R\alpha_{N,q}^s\,\mathcal{D}_s\rho_G,
    \qquad\text{where}\quad \rho_G(x)\coloneqq\frac{1}{\sqrt{2\pi}}\e^{-\frac{x^2}{2}},
    \end{equation}
    which is a real Schwartz function of total mass one.
\end{lemma}
We are now ready to state our first main result.
\begin{theorem}\label{theo:locallaw}
    Let $\varepsilon\in(0,\frac{1}{2})$, $q$ be an even integer such that $2\leqslant q \leqslant N^{\frac{1}{2}-\varepsilon}$, $L>0$, $R\geqslant0$ an integer, and $\tau\in\{\text{Full},+,-\}$. Set
    \[
    \Dcal_{L,\varepsilon}= \left\{z=E+\i\eta:\,\,\vert E\vert \leqslant L,\,\, (\log N)^{\varepsilon}\alpha_{N,q}^{\frac{1}{2}}\leqslant \eta\leqslant 1\right\}.
    \]
    There exist absolute constants $c,C>0$ and a constant $C_R>0$, depending only on $L$ and $R$, such that for every $z=E+\i\eta\in \Dcal_{L,\varepsilon}$ and every $0<t\leqslant1$ we have
    \begin{equation}
    \label{eq:llawave}
    \Pds\left(
        \left\vert s_{N,\tau}(z)-\Phi_{R}(z)\right\vert > \frac{C_R\,\alpha_{N,q}^{R+1}}{\eta^{2R+2}}+t
    \right)
    \leqslant C\exp\left(-c\,M\eta^{3}t^2\right)
    \end{equation}
    for $N$ large enough. Moreover, taking a union bound over a fine net of $\Dcal_{L,\varepsilon}$, we obtain the uniform estimate
    \begin{equation}
    \label{eq:llawuniform}
    \Pds\left(
        \exists\,z=E+\i\eta\in\Dcal_{L,\varepsilon}:\,\,
        \left\vert s_{N,\tau}(z)-\Phi_{R}(z)\right\vert
        \geqslant
        C_R\left(
            \frac{\alpha_{N,q}^{R+1}}{\eta^{2R+2}}
            +
            \frac{\sqrt{\log M}}{\sqrt{M}\,\eta^{\frac32}}
        \right)
    \right)\leqslant M^{-10}.
    \end{equation}
    \end{theorem}
 \begin{remark}
    The deterministic profile $\Phi_R$ can be replaced by a different profile $\widetilde{\Phi}_R$, which is now a Stieltjes transform of a probability measure, at the price of a further error $C_Re^{-c_r\alpha_{N,q}^{-1/2}}\eta^{-1}$. More precisely, by explicit computations is possible to see that $\widetilde{\Phi}_R$ is nothing else than the Stieltjes transofm of the positive part of $\rho_R$.
\end{remark} 

     We point out that the deterministic error in \eqref{eq:llawave} and \eqref{eq:llawuniform} is in fact uniform in $E\in\Rbb$, by Proposition~\ref{pro:higherstab} below; the parameter $L$ enters only through the cardinality of the net used in the union bound leading to \eqref{eq:llawuniform}. We shall use this uniformity over the whole real line in the proof of Lemma~\ref{lem:sff1}.
    The two error terms in \eqref{eq:llawuniform} play different roles: the first one is the accuracy of the deterministic profile $\Phi_R$, while the second one accounts for the fluctuation of $s_{N,\tau}$ around its expectation, and dominates as soon as $\alpha_{N,q}^{R+1}\eta^{-2R-\frac12}\leqslant M^{-\frac{1}{2}}\sqrt{\log M}$. This is the reason why we formulated \eqref{eq:llawave} with a free parameter $t$ rather than with a single error term. Note also that on $\Dcal_{L,\varepsilon}$ we have $\alpha_{N,q}\eta^{-2}\leqslant(\log N)^{-2\varepsilon}$, so that the deterministic error in \eqref{eq:llawave} is indeed small, for every fixed $R$, on the whole of $\Dcal_{L,\varepsilon}$.
    \begin{remark}
    \label{rem:extension}
      For fixed $q$ and $\eta\sim1$, the deterministic error in \eqref{eq:llawave} is of order $N^{-R-1}$, for an arbitrarily large $R>0$, while the fluctuation term in \eqref{eq:llawuniform} is of order $M^{-\frac12}\sqrt{\log M}\asymp N^{-\frac q2}\sqrt{\log N}$; the two become comparable when $R\approx \frac{q}{2}-1$. Thus, the larger $q$ becomes, the larger $R$ should be taken in the deterministic profile for it to be relevant.
    \end{remark}
    \eer

    As a first corollary, we obtain a quantitative description of the density of states on mesoscopic scales, in the form of a counting estimate for the eigenvalues in a mesoscopic interval. Since it is a classical consequence of the local law, we omit the proof. To lighten the notation we denote
    \begin{equation}
    \label{eq:deferrorR}
    \Efr_R(\eta)\coloneqq \frac{\alpha_{N,q}^{R+1}}{\eta^{2R+2}}+\frac{\sqrt{\log M}}{\sqrt{M}\,\eta^{\frac32}}
    \end{equation}
    \begin{corollary}
    \label{cor:counting}
        Let $\varepsilon\in(0,\frac{1}{2})$, $L>0$, let $R\geqslant0$ be an integer, let $q$ be even with $2\leqslant q\leqslant N^{\frac{1}{2}-\varepsilon}$, and let $\tau\in\{\text{Full},+,-\}$. There is $C_R>0$, depending only on $L$ and $R$, such that, for $N$ large enough, the following estimates hold with probability at least $1-M^{-10}$:
       for every $\eta$ and $w$ with $(\log N)^{\varepsilon}\alpha_{N,q}^{\frac{1}{2}}\leqslant\eta\leqslant w\leqslant1$ and $\Efr_R(\eta)\leqslant1$, and every $E\in[-L,L]$,
        \begin{equation}
        \label{eq:countingquant}
        \left\vert\frac{\vert \{i:\lambda_i^\tau \in [E-w,E+w]\}\vert}{2D^\tau w}
        -
        \frac{1}{2w}\int_{E-w}^{E+w} \varrho_R(x)\d x
        \right\vert
        \leqslant
        C_R\left(
        \frac{\eta}{w}\log\frac{4}{\eta}
        +\Efr_R(\eta)
        \right),
        \end{equation}
        with $\varrho_R$ the profile density of \eqref{eq:defvarrhoR}.
    \end{corollary}
    The second term is the local law error, while the first one is the cost of resolving a sharp interval at scale $\eta$. The reference here is the profile density $\varrho_R$ itself, but the refinement it carries is invisible at this resolution: since $\eta\geqslant\alpha_{N,q}^{\frac12}$ on $\Dcal_{L,\varepsilon}$, the first term is at least $\alpha_{N,q}^{\frac12}w^{-1}$ and therefore always dominates $\Vert\varrho_R-\rho_G\Vert_\infty\leqslant C_R\alpha_{N,q}$, so that $\varrho_R$ may be replaced by the Gaussian density $\rho_G$ at no cost. The additional gain coming from the higher order profiles in \eqref{eq:countingquant} is thus not the reference but the \emph{rate}. Balancing the two by the choice $\eta=\big(w\,\alpha_{N,q}^{R+1}\big)^{\frac{1}{2R+3}}$ yields the error $\big(\alpha_{N,q}^{R+1}w^{-2R-2}\big)^{\frac{1}{2R+3}}\log N$, up to the fluctuation term. Concretely, for fixed $q\geqslant4$ and $w\sim1$ this choice is admissible and gives, with probability at least $1-M^{-10}$,
    \begin{equation}
    \label{eq:countingopt}
    \max_{\vert E\vert\leqslant L}
    \left\vert\frac{\vert \{i:\lambda_i^\tau \in [E-w,E+w]\}\vert}{2D^\tau w}
    -
    \frac{1}{2w}\int_{E-w}^{E+w} \varrho_R(x)\d x
    \right\vert
    \leqslant
    C_R\,N^{-\frac{R+1}{2R+3}}\log N,
    \end{equation}
    which improves from $N^{-\frac13}\log N$ at $R=0$ to any power of $N$ larger than $N^{-\frac12}$ as $R$ grows; the exponent $\frac12$ is the natural barrier here, since $\Dcal_{L,\varepsilon}$ does not reach below the scale $\alpha_{N,q}^{\frac12}$.
    \eer

    As another corollary, we can control the spectral form factor up to times of order $\sqrt{\log N}$. Here the deterministic profile enters again in an essential way, both through the reference function and through the length of the admissible time interval. By Lemma~\ref{lem:thetastruct}, \eqref{eq:defvarrhoR} and $\widehat{\rho_G^{(k)}}(t)=(\i t)^k\e^{-\frac{t^2}{2}}$, the Fourier transform of the profile density is
    \begin{equation}
    \label{eq:hatvarrho}
    \widehat{\varrho_R}(t)\coloneqq\int_\Rbb \e^{-\i tE}\varrho_R(E)\d E=\e^{-\frac{t^2}{2}}\,\pi_R(t),
    \qquad
    \pi_R(t)\coloneqq 1+\sum_{s=1}^R\alpha_{N,q}^s\,p_s(t),
    \end{equation}
    where $p_s$ is the explicit real even polynomial, with $2s+2\leqslant\deg p_s\leqslant 4s$, obtained from $\Theta_s=\mathcal{D}_sm$ by replacing every $\partial^k$ by $(\i t)^k$.
\begin{corollary}
\label{cor:sff}
Let $\varepsilon_0>0$, $q$ be even with $2\leqslant q\leqslant N^{\frac{1}{2}-\varepsilon_0}$, and let $R\geqslant0$ be an integer.
For $\tau\in\{\text{Full},+,-\}$, define the spectral form factor
\[
K_{N,\tau}(t)
=
\Eds\left[
\left|
\tr_\tau(\e^{-\i tH_q})
\right|^2
\right]\footnote{Here $\tr_\tau$ is the normalized trace, so that $K_{N,\tau}(0)=1$; the spectral form factor in the physics literature is often defined as $\Eds\vert \Tr(\e^{-\i tH_q})\vert^2$, i.e. it is $(D^\tau)^2$ times our definition.}.
\]
\eer
There is $C_R>0$, depending only on $R$, such that for every $\delta\in(0,1)$ the following holds for $N$ large enough. Let $T$ satisfy
\begin{equation}
\label{eq:sfftimerange}
1\leqslant T\leqslant
\sqrt{(1-\delta)\min\left(\frac{4(R+1)}{5}\log\frac{1}{\alpha_{N,q}},\,\log M\right)}.
\end{equation}
Then
\begin{equation}
\label{eq:boundsff}
\sup_{|t|\leqslant T}
\left|
K_{N,\tau}(t)-\left\vert\widehat{\varrho_R}(t)\right\vert^2
\right|
\leqslant
C_R\left[
\left(
\alpha_{N,q}^{R+1}T^{2R+2}
\right)^{\frac{4}{5}}
+
\frac{T^2}{M}
\right]
\leqslant
C_R\left[
\left(
\alpha_{N,q}\log\frac{1}{\alpha_{N,q}}
\right)^{\frac{4(R+1)}{5}}
+
\frac{\log M}{M}
\right].
\end{equation}
\end{corollary}
\begin{remark}
\label{rem:sffrange}

The two terms in the minimum in \eqref{eq:sfftimerange} have different origins, and only the first one is affected by $R$. It comes from the accuracy $\alpha_{N,q}^{R+1}$ of the deterministic profile. The second one comes from the variance of $\tr_\tau(\e^{-\i tH_q})$, which is of order $T^2M^{-1}$ by \eqref{eq:sffvariance} and is unrelated to the local law. Since $\log M\asymp q\log\frac Nq$ while $\log\frac{1}{\alpha_{N,q}}\asymp\log\frac{N}{q^2}$, the gain thus saturates at $T\asymp\sqrt{q\log N}$ for fixed $q$, when $R$ reaches $\frac{5q}{4}$ or so, and the admissible times remain logarithmic. We are in particular still very far from the Heisenberg time $T\asymp2^{\frac N2}$, at which the spectral form factor is expected to exhibit the local \eer random-matrix behaviour.
 \end{remark}

    While Theorem~\ref{theo:locallaw} hold for $q\ll \sqrt{N}$. We now restrict to fixed $q$, i.e. independent of $N$, and show a stronger delocalization bound for single eigenvectors. This will be a corollary of the following isotropic local law:
    \begin{theorem}
\label{theo:isolaw}
Fix any large $C_*>0$ and any arbitrary small $\delta\in(0,\frac{1}{2})$. Let $q$ be an even integer such that $4\leqslant q\leqslant C_*$, let $L>0$, fix $\tau\in\{+,-\}$, let $R\geqslant0$ be an integer, and let ${\bm x},{\bm y}\in\Hcal_N^\tau$ be deterministic unit vectors. Then, there are constants $C,c,C_L>0$, depending only on $L,\delta$, $R$, and $C_*$, such that, for all sufficiently large even $N$, every $0<t\leqslant 1$, and every $N^{-\frac{1}{2}+\delta}\leqslant\eta\leqslant 1$,
we have
\begin{equation}
\label{eq:isotropicmain}
\mathbb{P}\left(
\max_{|E|\leqslant L}
\left|
\left\langle {\bm x},
\left(G_\tau(E+\i\eta)-\Phi_{R}(E+\i\eta)\Id_{\Hcal_N^\tau}\right)
{\bm y}
\right\rangle
\right|
\geqslant
\frac{C\alpha_{N,q}^{R+1}}{\eta^{2R+2}}+t
\right)
\leqslant
C_L\left(1+\frac{1}{\eta^2t}\right)
\exp\left(-cN^{\frac{q}{2}}\eta^2t^2\right),
\end{equation}
where $\Phi_R$ is the deterministic profile of Definition~\ref{def:thetas}. In particular, since $q$ is fixed and $\eta\geqslant N^{-\frac{1}{2}+\delta}$, the deterministic error above is at most $CN^{-2\delta(R+1)}$, i.e. an arbitrarily small power of $N$.
\end{theorem}

As a corollary of this result we obtain eigenvector delocalization in any deterministic basis:
\begin{corollary}
\label{theo:delocinfinity}
Under the assumptions of Theorem~\ref{theo:isolaw}, for any deterministic orthonormal basis $({\bm u}_a^\tau)_{1\leqslant a\leqslant D^\tau}$ of $\Hcal_N^\tau$, we have
\begin{equation}
\label{eq:gooddeloc}
\mathbb{P}\left(
\max_{j:|\lambda_j^{(\tau)}|\leqslant L}
\max_{1\leqslant a\leqslant D^\tau}
\left|\left\langle {\bm u}_a^\tau,\psi_j^{(\tau)}\right\rangle\right|^2
\geqslant Cu
\right)
\leqslant
\exp\left(-cN^{\frac{q}{2}}u^2\right),
\end{equation}
for $N^{-\frac{1}{2}+\delta}\leqslant u\leqslant 1$.
\end{corollary}

Note that for every fixed even $q\geqslant 4$, the delocalization scale in Corollary~\ref{theo:delocinfinity} is $N^{-\frac{1}{2}+\delta}$. The isotropic local law and delocalization estimates in the two parity sectors automatically imply the corresponding statements in the full space.

\begin{remark}
In the occupation basis, the Gaussian rotation invariance of the SYK Hamiltonian yields a substantially stronger delocalization estimate. Let $(e_x)_{x\in\Bcal_\tau}$ denote the occupation basis of $\Hcal_N^\tau$. For every $\varepsilon\in(0,1)$, there exists $c_\varepsilon>0$ such that, for $N$ sufficiently large,
\[
    \Pds\left(
    \max_{1\leqslant j\leqslant D^\tau}
    \max_{x\in\Bcal_\tau}
    |\langle e_x,\psi_j^\tau\rangle|^2
    >(D^\tau)^{-1+\varepsilon}
    \right)
    \leqslant\e^{-c_\varepsilon N},
\]
uniformly in even $2\leqslant q\leqslant N-2$. This estimate holds for all eigenvalues and simultaneously for every choice of orthonormal eigenbasis, including within degenerate eigenspaces. Its proof uses symmetry invariance and moment
estimates for the orbit of occupation vectors, independently of the local
law; we omit the details as they are quite different from our approach here.
The restriction to the occupation basis is essential to this argument,
whereas Corollary~\ref{theo:delocinfinity} applies to any deterministic
orthonormal basis.
\end{remark}

\subsection*{Acknowledgments} The research L. B. is supported in part by a Natural Sciences and Engineering Research Council of Canada (NSERC) RGPIN 2023-03882 \& DGECR 2023-00076 and a FRQNT Etablissement de la Relève Professorale 364387. The research of G. C. is supported by the Italian Ministry of University and Research (MUR) - Fondo Italiano per la Scienza (FIS3) - 2024 Call (FIS-2024-07468), project UBLOCO, CUP F53C25000940001, and also partially supported by the MUR Excellence Department Project MatMod@TOV awarded to the Department of Mathematics, University of Rome Tor Vergata, CUP E83C23000330006. Additionally, G.C. thanks UMI (Unione Matematica Italiana), INdAM (Istituto Nazionale di Alta Matematica “Francesco Severi”), and the group GNFM. Large language models were used to assist with some mathematical arguments, literature reviews, and manuscript checking.

\section{Proof of the local law}
\label{sec:llaw}

We start by computing the approximate differential equation followed by $m_N(z)$ through the following lemma.
\begin{lemma}\label{lem:ibp}
    Let $z\in\Cbb_+$ and $\tau\in\{\text{Full},+,-\}$ then we have
    \[
    m_{N,\tau}'(z)+zm_{N,\tau}(z)+1=-\Eds\left[
       \tr_\tau\left(
            G_\tau(z)\left(
                \Scal[G_\tau(z)]-G_\tau(z)
            \right)
        \right)
    \right]
    \eqqcolon-\Rcal_\tau(z)
    \]
    where we defined the covariance operator
    \[
    \Scal[R] = \frac{1}{M}\sum_{\substack{A\subset\unn{1}{N}\\\vert A\vert=q}}\Gamma_A R \Gamma_A.
    \]
\end{lemma}
\begin{proof}
    We omit the $\tau$-dependence and work only in the full sector. The proof in the other parity sectors is the same. We start with the relation $H_qG(z) = \Id_D+zG(z)$ by definition of the resolvent which gives
    \[
    1+zm_N(z)=\Eds\left[\tr(H_qG(z))\right].
    \]
    We then use Gaussian integration by parts to write
    \begin{align*}
    \Eds\left[\tr(H_qG(z))\right]
    =
    \frac{1}{\sqrt{M}}\sum_{\substack{A\subset\unn{1}{N}\\\vert A\vert =q}}\Eds\left[J_A\tr(\Gamma_AG(z))\right]
    &=
    -\frac{1}{M}\sum_{\substack{A\subset\unn{1}{N}\\\vert A\vert =q}}
    \Eds\left[\tr(\Gamma_AG(z)\Gamma_AG(z))\right]\\&=-\Eds\left[\tr(G(z)\Scal[G(z)])\right]
    \end{align*}
    where we used the fact that $\Eds[J_AG(z)]=-\frac{1}{\sqrt{M}}\Eds\left[G(z)\Gamma_AG(z)\right].$ We finish the proof by noticing that $m'_N(z)=\Eds\left[\tr(G(z)^2)\right].$
\end{proof}

Throughout this section we abbreviate $\alpha\coloneqq\alpha_{N,q}$ inside proofs, and we omit the $\tau$-dependence whenever the argument is identical in the three sectors. We shall constantly use the following notation. For a $q$-subset $A$ we set
\[
H^{(A)}_q = \Gamma_AH_q\Gamma_A,\quad
G_A(z) = \Gamma_AG(z)\Gamma_A = \left(H^{(A)}_q-z\Id_D\right)^{-1},\quad
\Ical_A = \left\{B\subset\unn{1}{N}:\vert B\vert =q,\,\Gamma_A\Gamma_B=-\Gamma_B\Gamma_A\right\},
\]
so that $\vert \Ical_A\vert = M\alpha_{N,q}$ by the definition \eqref{eq:alphanqasymp} of $\alpha_{N,q}$; since $q$ is even, $\Gamma_A$ preserves the two parity sectors, so that the same definitions make sense in each of them. Since conjugating by $\Gamma_A$ changes the sign of exactly those terms of $H_q$ which are indexed by $\Ical_A$, the difference
\begin{equation}\label{eq:defdeltaa}
\Delta_A \coloneqq H_q-H_q^{(A)}
\quad\text{is given by}\quad
\Delta_A = \frac{2}{\sqrt{M}}\sum_{B\in\Ical_A}J_B\Gamma_B,
\quad\text{and satisfies}\quad
\Eds\left[\tr(\Delta_A^2)\right]=4\alpha_{N,q},
\end{equation}
where the last identity follows from $\Gamma_B^2=\Id_D$ and the fact that only the diagonal terms survive. Note also that, by the definition of $\Scal$ and by cyclicity of the trace, the remainder of Lemma~\ref{lem:ibp} may be rewritten as
\begin{equation}
\label{eq:sourceasdiff}
\Rcal(z)=\frac{1}{M}\sum_{\substack{A\subset\unn{1}{N}\\\vert A\vert=q}}\Eds\left[\tr\left(G(z)\left(G_{A}(z)-G(z)\right)\right)\right].
\end{equation}
Finally, we will repeatedly use the Ward identity in the following form: if $\lambda_1\leqslant\dots\leqslant\lambda_{D^\tau}$ denote the eigenvalues of $H_q^\tau$ and $k\geqslant2$ is an integer, then, since $\Vert G_\tau(z)\Vert\leqslant\eta^{-1}$,
\begin{equation}\label{eq:boundG4}
\tr_\tau\left(\vert G_\tau(E+\i\eta)\vert^{k}\right)
\leqslant \frac{1}{\eta^{k-2}}\,\tr_\tau\left(\vert G_\tau(E+\i\eta)\vert^{2}\right)
=\frac{1}{\eta^{k-2}}\cdot\frac{1}{D^\tau}\sum_{i=1}^{D^\tau} \frac{1}{(\lambda_i-E)^2+\eta^2}
=\frac{\Im s_{N,\tau}(z)}{\eta^{k-1}}.
\end{equation}
\eer

Before estimating the remainder $\Rcal_\tau$, we collect a few elementary properties of the Stieltjes transform $m$ of the standard Gaussian law which will be used repeatedly in the sequel.
\begin{lemma}
\label{lem:gaussprofile}
    Let $\rho_G(x)=(2\pi)^{-\frac{1}{2}}\e^{-\frac{x^2}{2}}$ denote the standard Gaussian density. For every integer $k\geqslant0$ and every $z\in\Cbb_+$ we have
    \begin{equation}
    \label{eq:mkrepr}
    m^{(k)}(z)=k!\int_\Rbb \frac{\rho_G(x)}{(x-z)^{k+1}}\d x=\int_\Rbb\frac{\rho_G^{(k)}(x)}{x-z}\d x,
    \end{equation}
    so that $m^{(k)}$ is the Stieltjes transform of $\rho_G^{(k)}$. Moreover, with $L$ as in \eqref{eq:defL},
    \begin{equation}
    \label{eq:gaussode}
    L[m](z)=-1
    \qquad\text{and}\qquad
    L\left[m^{(k)}\right](z)=m^{(k+1)}(z)+zm^{(k)}(z)=-k\,m^{(k-1)}(z)
    \quad\text{for }k\geqslant1,
    \end{equation}
    and consequently, for every integer $j\geqslant0$, the function $-\frac{1}{j+1}m^{(j+1)}$ is the unique solution of $L[h]=m^{(j)}$ which vanishes as $\Im z\to\infty$.
    Finally, for every integer $k\geqslant0$ and every $z=E+\i\eta\in\Cbb_+$,
    \begin{equation}
    \label{eq:gaussderivbound}
    \left\vert m^{(k)}(z)\right\vert \leqslant \frac{k!}{\eta^{k+1}},
    \end{equation}
    and $m$ extends holomorphically to the half-plane $\{z\in\Cbb:\Im z>-3\}$, so that for every integer $k\geqslant0$ there exists $C_k>0$ with
    \begin{equation}
    \label{eq:gaussderivunif}
    \left\vert m^{(k)}(E+\i\eta)\right\vert \leqslant C_k
    \qquad\text{for every}\quad E\in\Rbb\quad\text{and}\quad 0<\eta\leqslant2.
    \end{equation}
\end{lemma}
The proof of Lemma~\ref{lem:gaussprofile} is elementary and we 
postpone 
it 
to Appendix~\ref{app:techres}. 
We can already deduce from it the structure of the deterministic profiles of Definition~\ref{def:thetas}.

\begin{proof}[Proof of Lemma~\ref{lem:thetastruct}]
    We argue by induction on $s$, the case $s=0$ being trivial since $\Theta_0=m^{(0)}$. Assume that, for every $1\leqslant r\leqslant s$, the function $\Theta_{s-r}$ is a linear combination of the $m^{(k)}$ with $k$ even and $2(s-r)+2\leqslant k\leqslant 4(s-r)$, with the convention that only $k=0$ occurs when $s=r$. Then $\Theta_{s-r}^{(2r+1)}$ is a combination of the $m^{(k+2r+1)}$, and the last statement in \eqref{eq:gaussode} shows that the unique solution of \eqref{eq:thetarec} vanishing at infinity is
    \begin{equation}
    \label{eq:thetaexpl}
    \Theta_s=-\sum_{r=1}^{s}\frac{c_r}{(2r+1)!}\sum_{k=2(s-r)+2}^{4(s-r)}a^{(s-r)}_k\,\frac{m^{(k+2r+2)}}{k+2r+2},
    \qquad\text{where}\quad \Theta_{s-r}=\sum_{k=2(s-r)+2}^{4(s-r)}a^{(s-r)}_k\,m^{(k)}.
    \end{equation}
    Every index $k+2r+2$ occurring here is even, is equal to $2s+2$ when $r=s$ and at least $2(s-r)+2+2r+2=2s+4$ when $r<s$, and is at most $4(s-r)+2r+2=4s-2r+2\leqslant4s$ for $r\geqslant1$. This proves the stated structure. 

    The bounds \eqref{eq:thetabounds} follow from \eqref{eq:gaussderivunif} and \eqref{eq:gaussderivbound} respectively, using that $\Theta_s$ only involves derivatives of $m$ of order at least $2s+2$. Finally, writing $\Theta_s=\mathcal{D}_sm$ and using the second identity in \eqref{eq:mkrepr}, we get that $\Theta_s$ is the Stieltjes transform of $\mathcal{D}_s\rho_G$, so that $\Phi_R$ is the Stieltjes transform of the function $\varrho_R$ of \eqref{eq:defvarrhoR}. The latter is a real Schwartz function since $\rho_G$ is, its total mass is $\int_\Rbb\rho_G=1$ because $\mathcal{D}_s$ only involves derivatives of order at least $2s+2\geqslant1$, whose integrals over $\Rbb$ vanish.
\end{proof}

We now turn to the expansion of the remainder $\Rcal_\tau$ of Lemma~\ref{lem:ibp}. The mechanism behind it is that each occurrence of $\Delta_A$ inside a trace of resolvents can be integrated by parts, which produces an additional factor $\alpha_{N,q}$. This is the content of the following lemma, which we shall apply to traces containing an arbitrary number of resolvents.
\begin{lemma}
\label{lem:multiins}
    Let $z=E+\i\eta\in\Cbb_+$ be such that $\alpha_{N,q}\leqslant\eta^2$, let $k\geqslant2$ and $j\geqslant1$ be integers, let $A_1,\dots,A_j$ be $q$-subsets, not necessarily distinct, and let $\tau\in\{\text{Full},+,-\}$. Let $\Xi$ be a product consisting of $k$ factors of the form $UG_\tau(z)U^\ast$, where the deterministic unitary matrices $U$ may vary from factor to factor, of the $j$ factors $\Delta_{A_1},\dots,\Delta_{A_j}$, and of arbitrary deterministic unitary matrices in between. Then
    \[
    \left\vert \Eds\left[\tr_\tau(\Xi)\right]\right\vert
    \leqslant
    \frac{C_{k,j}\,\alpha_{N,q}^{\left\lceil \frac j2\right\rceil}}{\eta^{\,k-1+(j\bmod 2)}}\,\Im m_{N,\tau}(z).
    \]
\end{lemma}
Let us comment on the form of $\Xi$ before giving the proof. The unitary conjugations, as well as the unitary matrices allowed between consecutive factors, are needed because the traces to which we shall apply the lemma arise from the resolvent expansion of $G_A=\Gamma_AG\Gamma_A$ and from repeated differentiations of it, and therefore contain resolvents conjugated by Clifford monomials, as well as such monomials sitting between consecutive resolvents. Since the Schatten norms are unitarily invariant, allowing them costs nothing in the estimate.
\begin{proof}
    We omit the $\tau$-dependence. Expanding every factor $\Delta_{A_i}$ by \eqref{eq:defdeltaa} we obtain
    \[
    \Eds\left[\tr(\Xi)\right]
    =
    \frac{2^j}{M^{\frac j2}}\sum_{B_1\in\Ical_{A_1}}\cdots\sum_{B_j\in\Ical_{A_j}}
    \Eds\left[J_{B_1}\cdots J_{B_j}\tr\left(\Xi_{B_1,\dots,B_j}\right)\right],
    \]
    where $\Xi_{B_1,\dots,B_j}$ is obtained from $\Xi$ by replacing the factor $\Delta_{A_i}$ with $\Gamma_{B_i}$ for every $i$. Iterated Gaussian integration by parts gives
    \[
    \Eds\left[J_{B_1}\cdots J_{B_j}F\right]
    =
    \sum_{P}\left(\prod_{\{i,i'\}\in P}\delta_{B_iB_{i'}}\right)
    \Eds\left[\Big(\prod_{i\notin V(P)}\partial_{J_{B_i}}\Big)F\right],
    \]
    the sum running over all partial pairings $P$ of $\unn1j$, with $V(P)$ the set of indices covered by $P$; the number of such $P$ is bounded by a constant depending only on $j$. Fix such a $P$, with $a$ pairs, so that $b\coloneqq j-2a$ derivatives act on the trace. For any deterministic unitary matrix $U$ we have
    \[
    \partial_{J_B}\left(UG(z)U^\ast\right)
    =
    -\frac{1}{\sqrt{M}}\left(UG(z)U^\ast\right)\left(U\Gamma_BU^\ast\right)\left(UG(z)U^\ast\right),
    \]
    so that each differentiation produces a factor $M^{-\frac12}$ and replaces one resolvent factor by two factors of the same form, separated by a deterministic unitary matrix. Hence, after the $b$ differentiations, we are left with at most $C_{k,b}$ normalized traces of products of $k+b$ factors of the form $UG(z)U^\ast$ separated by deterministic unitary matrices. By the noncommutative H\"{o}lder inequality with all exponents equal to $k+b$ (the Schatten norms $\Vert X\Vert_p\coloneqq\tr(\vert X\vert^p)^{1/p}$ being normalized, consistently with $\tr$), and by the unitary invariance of the Schatten norms, each of these traces is bounded in modulus by
    \[
    \Vert G(z)\Vert_{k+b}^{k+b}=\tr\left(\vert G(z)\vert^{k+b}\right)\leqslant \frac{\Im s_N(z)}{\eta^{k+b-1}},
    \]
    where in the last step we used the Ward identity \eqref{eq:boundG4}. Moreover the Kronecker deltas leave $j-a$ free summation indices, each running over a set of cardinality at most $M\alpha$. Altogether, the contribution of $P$ is bounded by
    \[
    C_{k,j}\,\frac{1}{M^{\frac j2}}\,\left(M\alpha\right)^{j-a}\,\frac{1}{M^{\frac b2}}\,\frac{\Im m_N(z)}{\eta^{k+b-1}}
    =
    C_{k,j}\,\alpha^{j-a}\,\frac{\Im m_N(z)}{\eta^{k+j-2a-1}},
    \]
    since $-\frac j2+(j-a)-\frac{j-2a}{2}=0$. Writing $\alpha^{j-a}\eta^{-(k+j-2a-1)}=\left(\alpha\eta^{-2}\right)^{j-a}\eta^{\,j-k+1}$ and using that $\alpha\leqslant\eta^2$, this quantity is a decreasing function of $j-a$, and is therefore maximal when the number $a$ of pairs is maximal. Since a partial pairing of $\unn1j$ has at most $\lfloor j/2\rfloor$ pairs, the extremal case is $a=\lfloor j/2\rfloor$, for which $j-a=\lceil j/2\rceil$ and
    \[
    k+j-2a-1=k-1+\left(j-2\left\lfloor \tfrac j2\right\rfloor\right)=k-1+(j\bmod2),
    \]
    which gives the claim.
\end{proof}

\begin{proposition}
\label{pro:trunc}
    Let $R\geqslant1$, let $\tau\in\{\text{Full},+,-\}$, and let $z=E+\i\eta\in\Cbb_+$ be such that $\alpha_{N,q}\leqslant\eta^2$. Omitting the $\tau$-dependence and setting
    \[
    \Rcal^{(j)}(z)\coloneqq \frac1M\sum_{\substack{A\subset\unn{1}{N}\\\vert A\vert = q}}\Eds\left[\tr\left(G^2\left(\Delta_AG\right)^j\right)\right],
    \]
    we have
    \begin{equation}
    \label{eq:trunc}
    \Rcal(z)=\sum_{j=1}^{2R}\Rcal^{(j)}(z)+\O{\frac{\alpha_{N,q}^{R+1}}{\eta^{2R+3}}\Im m_{N}(z)},
    \qquad\text{and}\qquad
    \left\vert \Rcal^{(j)}(z)\right\vert \leqslant \frac{C_j\,\alpha_{N,q}^{\left\lceil \frac j2\right\rceil}}{\eta^{\,2\left\lceil \frac j2\right\rceil+1}}\Im m_{N}(z)
    \end{equation}
    for every $1\leqslant j\leqslant 2R$. 
\end{proposition}
\begin{proof}
    Iterating the resolvent identity $G_A=G+G_A\Delta_AG$ we obtain
    \[
    G_A-G=\sum_{j=1}^{2R}G\left(\Delta_AG\right)^j+G_A\left(\Delta_AG\right)^{2R+1},
    \]
    so that, by \eqref{eq:sourceasdiff} and cyclicity of the trace,
    \[
    \Rcal(z)=\sum_{j=1}^{2R}\Rcal^{(j)}(z)+\frac1M\sum_{\substack{A\subset\unn{1}{N}\\\vert A\vert = q}}\Eds\left[\tr\left(GG_A\left(\Delta_AG\right)^{2R+1}\right)\right].
    \]
    The trace defining $\Rcal^{(j)}$ contains $k=j+2$ resolvent factors and $j$ insertions, while the last trace contains $k=2R+3$ resolvent factors and $j=2R+1$ insertions. Lemma~\ref{lem:multiins} gives the two estimates in \eqref{eq:trunc}, using that $k-1+(j\bmod2)=2\lceil j/2\rceil+1$ when $k=j+2$. \cp 
\end{proof}

We now prove stability of the differential equation at leading order. This is the case $R=0$ of Proposition~\ref{pro:higherstab}, and it is where the induction proving the latter starts.
\begin{lemma}\label{lem:expect}
    Let $z=E+\i\eta$ with $E\in\Rbb$ and $\eta\in(0,1]$. There exist absolute constants $c,C>0$ such that if $\alpha_{N,q}\leqslant c\,\eta^{2}$ then for $\tau\in\{\text{Full},+,-\}$
    \[
    \vert m_{N,\tau}(z)-m(z)\vert \leqslant \frac{C\alpha_{N,q}}{\eta^2}.
    \]
\end{lemma}
\begin{proof}
    We omit the $\tau$-dependence and abbreviate $\alpha\coloneqq\alpha_{N,q}$. We denote the difference
    $d(z) \coloneqq m_N(z)-m(z)$
    which is analytic on $\Cbb_+$ and, by Lemma~\ref{lem:ibp} and \eqref{eq:gaussode}, follows the equation
    \[
    d'(z)+zd(z) = -\Rcal(z).
    \]
    If we multiply the equation above by $\e^{\frac{z^2}{2}}$ we can rewrite it as
    \[
    \frac{\d}{\d z}\left(\e^{\frac{z^2}{2}}d(z)\right) = -\e^{\frac{z^2}{2}}\Rcal(z).
    \]
    For fixed $E$ and $\eta$, we consider $z_s = E+\i s$ for $s\geqslant \eta$ and define
    \[
    \Psi(s) = \e^{\frac{z_s^2}{2}}d(z_s)
    \quad\text{so that}\quad
    \Psi'(s) = \i\e^{\frac{z_s^2}{2}}(d'(z_s)+z_sd(z_s))=-\i\e^{\frac{z_s^2}{2}}\Rcal(z_s).
    \]
    Integrating over the vertical line, we obtain that for any $S>\eta$,
    \[
    \Psi(S)-\Psi(\eta) = -\i\int_{\eta}^S \e^{\frac{z_s^2}{2}}\Rcal(z_s)\d s.
    \]
 In addition, since $\vert \Psi(s)\vert \leqslant \frac{2}{s}\e^{\frac{E^2-s^2}{2}}\xrightarrow[s\to\infty]{}0$ using the fact that $m$ and $m_N$ are Stieltjes transforms, we can obtain the representation
    \[
    d(E+\i\eta)=\i\e^{-\frac{(E+\i\eta)^2}{2}}
    \int_\eta^{\infty}\e^{\frac{(E+\i s)^2}{2}}\Rcal(E+\i s)\d s
    \]
 taking the modulus, we obtain
    \begin{equation}\label{eq:bounddiff}
    \vert d(E+\i\eta)\vert \leqslant \e^{\frac{\eta^2}{2}}\int_\eta^\infty \e^{-\frac{s^2}{2}}\vert \Rcal(E+\i s)\vert \d s.
    \end{equation}
    We now use our bound on the remainder to obtain a bound on $\vert d(z)\vert.$ First denote
    \[
    S_\eta = 1+\sup_{\substack{E\in\Rbb\\\eta\leqslant s\leqslant 1}} \Im m_N(E+\i s)\leqslant 1+\frac{1}{\eta}.
    \]
    We note that we have $\Im m_N(E+\i s) \leqslant S_\eta$ for any $s\geqslant \eta$. Indeed, if $s\leqslant 1$ this is clear from the definition of $S_\eta$ and if $s\geqslant 1$ then we have
    \[
    \Im m_N(E+\i s) \leqslant \frac{1}{s}\leqslant 1\leqslant S_\eta.
    \]
    Since $\alpha\leqslant c\eta^2\leqslant s^2$ for every $s\geqslant\eta$,  Proposition~\ref{pro:trunc} with $R=1$ applies at every height $s\geqslant\eta$ and gives
    \[
    \vert \Rcal(E+\i s)\vert \leqslant \frac{C\alpha}{s^3}\,\Im m_N(E+\i s)\leqslant \frac{C\alpha\,S_\eta}{s^3}.
    \]
    Inserting this into \eqref{eq:bounddiff}, and using that $\e^{\frac{u^2-s^2}{2}}\leqslant1$ for $s\geqslant u$, we obtain for any $u\in[\eta,1]$ that
    \[
    \vert d(E+\i u)\vert \leqslant C\alpha S_\eta\int_u^\infty\frac{\d s}{s^3}=\frac{C\alpha S_\eta}{2u^2}
    \leqslant \frac{C\alpha S_\eta}{\eta^{2}}.
    \]
    Since the Gaussian density is bounded, $\Im m$ is bounded on $\Cbb_+$, and the bound above is uniform in $E$ and $u$; therefore
    \[
    S_\eta
    \leqslant C\left(
    1+\frac{\alpha}{\eta^{2}}S_\eta
    \right).
    \]
    Thus, choosing $c$ sufficiently small so that $C\alpha\eta^{-2}\leqslant\frac12$, we get $S_\eta \leqslant 2C$, which gives the final bound
    \[
    \vert d(E+\i\eta)\vert \leqslant \frac{C\alpha}{\eta^{2}},
    \]
    for some absolute constant $C>0$.
\end{proof}

We record separately the behaviour of $m_{N,\tau}-m$ at large spectral scales, where the differential equation can be integrated without any bootstrap.
\begin{lemma}
\label{lem:largeeta}
    There exists an absolute constant $C>0$ such that, for every $E\in\Rbb$, every $s>0$ with $\alpha_{N,q}\leqslant s^2$, and every $\tau\in\{\text{Full},+,-\}$, we have
    \[
    \left\vert m_{N,\tau}(E+\i s)-m(E+\i s)\right\vert \leqslant \frac{C\alpha_{N,q}}{s^3}.
    \]
\end{lemma}
\begin{proof}
    We omit the $\tau$-dependence, abbreviate $\alpha\coloneqq\alpha_{N,q}$ and set $d_0\coloneqq m_N-m$. For every $u\geqslant s$ we have $\Im m_N(E+\i u)\leqslant u^{-1}$ and $\alpha\leqslant s^2\leqslant u^2$, so that Proposition \ref{pro:trunc} gives
    \[
    \left\vert \Rcal(E+\i u)\right\vert \leqslant \frac{C\alpha}{ u^3}\,\Im m_N(E+\i u)\leqslant \frac{C\alpha}{u^4}
    \qquad\text{for every }u\geqslant s.
    \]
    By \eqref{eq:gaussode} and Lemma~\ref{lem:ibp}, the function $d_0$ satisfies $d_0'(z)+zd_0(z)=-\Rcal(z)$, so that, arguing exactly as in the derivation of \eqref{eq:bounddiff},
    \[
    \left\vert d_0(E+\i s)\right\vert \leqslant \e^{\frac{s^2}{2}}\int_s^\infty \e^{-\frac{u^2}{2}}\left\vert \Rcal(E+\i u)\right\vert \d u
    \leqslant C\alpha\int_s^\infty \frac{\d u}{u^4}=\frac{C\alpha}{3s^3},
    \]
    where we used that $\e^{\frac{s^2-u^2}{2}}\leqslant1$ for $u\geqslant s$.
\end{proof}

The two previous lemmas combine into the following a priori bounds, which we shall use repeatedly and which explain why the whole scheme only requires the smallness of $\alpha_{N,q}\eta^{-2}$.
\begin{lemma}
\label{lem:apriori}
    There exist absolute constants $c,C>0$ such that, for every $\eta\in(0,1]$ with $\alpha_{N,q}\leqslant c\,\eta^2$ and every $\tau\in\{\text{Full},+,-\}$,
    \begin{equation}
    \label{eq:apriorimn}
    \Im m_{N,\tau}(E+\i s)\leqslant C
    \qquad\text{for every }E\in\Rbb\text{ and }s\geqslant \frac{\eta}{2},
    \end{equation}
    and
    \begin{equation}
    \label{eq:aprioridiff}
    \left\vert m_{N,\tau}(w)-m(w)\right\vert \leqslant \frac{C\alpha_{N,q}}{(\Im w)^{2}}
    \qquad\text{for every }w\in\Cbb_+\text{ with }\frac{\eta}{2}\leqslant \Im w\leqslant 2.
    \end{equation}
\end{lemma}
\begin{proof}
    Let $c_0$ be the constant of Lemma~\ref{lem:expect} and set $c\coloneqq c_0/4$. For $s\geqslant\eta/2$ we then have $\alpha_{N,q}\leqslant c\eta^2\leqslant 4c\,s^2=c_0s^2$, so that Lemma~\ref{lem:expect} may be applied at every height $s\in[\eta/2,1]$ and gives $\Im m_{N,\tau}(E+\i s)\leqslant \Im m(E+\i s)+C\alpha_{N,q}s^{-2}\leqslant C+4Cc\leqslant C$, by \eqref{eq:gaussderivunif} with $k=0$. For $s\geqslant1$ we simply use $\Im m_{N,\tau}(E+\i s)\leqslant s^{-1}\leqslant1$, which proves \eqref{eq:apriorimn}. The bound \eqref{eq:aprioridiff} follows from Lemma~\ref{lem:expect} when $\Im w\leqslant1$, and from Lemma~\ref{lem:largeeta} when $1\leqslant\Im w\leqslant2$, since then $(\Im w)^{-3}\leqslant4(\Im w)^{-2}$.
\end{proof}

We can now expand the remainder $\Rcal_\tau$ to an arbitrary order in $\alpha_{N,q}$, which is the key input for the local law.
\begin{proposition}
\label{pro:highode}
    Let $R\geqslant1$. There is a constant $C_R>0$ such that $\vert c_r\vert \leqslant C_R$ for every $r\leqslant R$ and such that, uniformly for $z=E+\i\eta\in\Cbb_+$ with $\alpha_{N,q}\leqslant\eta^2$ and for $\tau\in\{\text{Full},+,-\}$,
    \begin{equation}
    \label{eq:highode}
    m_{N,\tau}'(z)+zm_{N,\tau}(z)+1=\sum_{r=1}^{R}\frac{c_r\,\alpha_{N,q}^r}{(2r+1)!}\,m_{N,\tau}^{(2r+1)}(z)+\Escr_{R,\tau}(z),
    \qquad
    \left\vert \Escr_{R,\tau}(z)\right\vert \leqslant \frac{C_R\,\alpha_{N,q}^{R+1}}{\eta^{2R+3}}\Im m_{N,\tau}(z),
    \end{equation}
    with $c_r$ as in \eqref{eq:crdef}. 
\end{proposition}
\begin{proof}
    We omit the $\tau$-dependence and abbreviate $\alpha\coloneqq\alpha_{N,q}$. By Proposition~\ref{pro:trunc} it suffices to expand each $\Rcal^{(j)}$, $1\leqslant j\leqslant 2R$, and to collect the terms of equal order in $\alpha$.

    Fix $j$, expand each of the $j$ factors $\Delta_A$ in $\Rcal^{(j)}$ by \eqref{eq:defdeltaa} and integrate by parts as in the proof of Lemma~\ref{lem:multiins}. This produces a sum over the partial pairings $P$ of $\unn1j$. A pairing with $a$ pairs contributes at order $\alpha^{r}$ with $r=j-a$ and, since $b=j-2a$ derivatives act on the trace and each of them creates one extra resolvent factor, it produces normalized traces of exactly
    \[
    (j+2)+(j-2a)=2(j-a)+2=2r+2
    \]
    factors of the form $UGU^\ast$; moreover, by Lemma~\ref{lem:multiins}, each of these terms is bounded by $C_{R}\,\alpha^{r}\eta^{-2r-1}\Im m_N(z)$. Hence \emph{every} contribution of order $\alpha^r$, and this for every value of $j$, carries exactly $2r+2$ resolvent factors and obeys this bound. In particular, since $\alpha\leqslant\eta^2$, all the contributions with $r\geqslant R+1$ are bounded by
    \[
    C_R\,\alpha^{r}\eta^{-2r-1}\Im m_N(z)=C_R\,\frac{\alpha^{R+1}}{\eta^{2R+3}}\left(\frac{\alpha}{\eta^2}\right)^{r-R-1}\Im m_N(z)\leqslant C_R\,\frac{\alpha^{R+1}}{\eta^{2R+3}}\Im m_N(z),
    \]
    and may be absorbed into the error term of \eqref{eq:highode}.

    It remains to see that each contribution of order $\alpha^r$, $1\leqslant r\leqslant R$, equals a bounded constant times $\Eds\left[\tr\left(G^{2r+2}\right)\right]$, up to terms of order $\alpha^{r+1}$. In each such trace the resolvent factors are of the form $UGU^\ast$ with $U$ a product of matrices $\Gamma_B$, where $B$ ranges over the summation indices.

    We first record that these matrices come in pairs. Fix a pairing $P$ with $a$ pairs, and consider a summation index $B_i$. If $i$ is covered by $P$, say $\{i,i'\}\in P$, then $B_i=B_{i'}$ and the two matrices $\Gamma_{B_i}$ and $\Gamma_{B_{i'}}$ already present in $\Xi$ carry the same index. If $i$ is not covered by $P$, then $\partial_{J_{B_i}}$ acts on the trace and, by the differentiation rule of the proof of Lemma~\ref{lem:multiins}, inserts a \emph{second} matrix $\Gamma_{B_i}$ next to the resolvent it splits. In either case each of the $j-a=r$ distinct indices occurs exactly twice, so that the trace contains exactly $2r$ matrices $\Gamma_B$ and $2r+2$ resolvent factors, in agreement with the counting above, and every $\Gamma_B$ has a well-defined partner.

    Using repeatedly the identity
    \[
    \Gamma_BG=G\Gamma_B+G_B\Delta_BG\Gamma_B,
    \]
    which is the resolvent identity $\Gamma_BG\Gamma_B=G_B=G+G_B\Delta_BG$ multiplied by $\Gamma_B$ on the right, one may bring every $\Gamma_B$ next to its partner and cancel the two using $\Gamma_B^2=\Id_D$. Each application of this identity replaces a product of $2r+2$ resolvent factors by the same product with one $\Gamma_B$ moved, plus a product of $2r+3$ resolvent factors carrying one additional insertion $\Delta_B$; integrating that insertion by parts turns the latter into products of $2r+4$ resolvent factors at order $\alpha^{r+1}$, in accordance with the counting above. What is left after all the $\Gamma_B$'s have been cancelled is $\Eds\left[\tr\left(G^{2r+2}\right)\right]$ multiplied by an average, over the summation indices, of products of the signs $\pm1$ produced by the reordering; each such average is bounded by $1$ in absolute value, and the number of terms is bounded by a constant depending only on $R$. This procedure terminates. Indeed, at a fixed order $\alpha^r$ the main term of each application of the identity above leaves the number of matrices $\Gamma_B$ unchanged and moves one of them one position closer to its partner among the at most $2r+2$ available positions, so that finitely many applications, in number bounded by a constant depending only on $R$, exhaust that order; and each error term raises the order in $\alpha$ by one, so that the outer recursion has depth at most $R+1$. We also stress that the signs produced by the reordering are pure Clifford combinatorics, determined by the indices $A,B_1,\dots,B_j$ and by the positions of the matrices, and in particular do \emph{not} depend on $z$; this is what allows the resulting constants to be pulled out of the trace and is what Lemma~\ref{lem:coeffmoments} below requires. Iterating the procedure and stopping as soon as the order in $\alpha$ exceeds $R$, we conclude that
    \[
    \Rcal(z)=-\sum_{r=1}^{R}\frac{\gamma_r\alpha^r}{(2r+1)!}\,m_N^{(2r+1)}(z)+\O{\frac{\alpha^{R+1}}{\eta^{2R+3}}\Im m_N(z)}
    \]
    for some real numbers $\gamma_r=\gamma_r(N,q)$ with $\vert\gamma_r\vert\leqslant C_R$, where we used that $\partial_z^{2r+1}G=(2r+1)!\,G^{2r+2}$ and therefore $\Eds\left[\tr\left(G^{2r+2}\right)\right]=\frac{m_N^{(2r+1)}(z)}{(2r+1)!}$. By Lemma~\ref{lem:ibp} this is \eqref{eq:highode} with $c_r$ replaced by $\gamma_r$, and Lemma~\ref{lem:coeffmoments} below shows that necessarily $\gamma_r=c_r$ for every $r\leqslant R$, which also proves that $\vert c_r\vert\leqslant C_R$. Applied in the sector $\tau$, Lemma~\ref{lem:coeffmoments} produces the constants built from the moments $\mu_{2k}^\tau$; these coincide with $\mu_{2k}$ for $k\leqslant R+1$ by the discussion following Definition~\ref{def:thetas}, so that the same $c_r$, and hence the same profile $\Phi_R$, serve the three sectors. Its hypothesis is met because $\Im m_{N,\tau}(E+\i\eta)\leqslant\eta^{-1}$ and $\alpha\leqslant\eta^2$ for $\eta\geqslant1$.

    We stress that the reordering above is only used to produce \emph{bounded} constants, which is all that we need; their explicit combinatorial evaluation, which already for $r=2$ involves averages of Clifford signs that are not determined by $\alpha_{N,q}$ alone, is entirely bypassed by Lemma~\ref{lem:coeffmoments}.
\end{proof}

\begin{lemma}
\label{lem:coeffmoments}
    Fix $N$ and $q$, let $R\geqslant1$, and assume that for some real numbers $\gamma_1,\dots,\gamma_R$ and some $C>0$ we have
    \[
    m_N'(z)+zm_N(z)+1=\sum_{r=1}^{R}\frac{\gamma_r\,\alpha_{N,q}^r}{(2r+1)!}m_N^{(2r+1)}(z)+\Escr(z),
    \qquad
    \vert \Escr(E+\i\eta)\vert \leqslant \frac{C\alpha_{N,q}^{R+1}}{\eta^{2R+4}},
    \]
    for all $E\in\Rbb$ and $\eta\geqslant1$. Then $\gamma_r=c_r$ for every $1\leqslant r\leqslant R$, with $c_r$ as in \eqref{eq:crdef}.
\end{lemma}
\begin{proof}
    We abbreviate $\alpha\coloneqq\alpha_{N,q}$ and $\mu_k\coloneqq\Eds\left[\tr\left(H_q^{k}\right)\right]$, which vanishes for $k$ odd since $H_q$ and $-H_q$ have the same law. For fixed $N$ and $q$ all these moments are finite, and so are $\mu_k^\ast\coloneqq\Eds\left[\tr\vert H_q\vert^{k}\right]$.
    %the couplings having moments of all orders.
Iterating the identity $(\lambda-z)^{-1}=-z^{-1}+\lambda z^{-1}(\lambda-z)^{-1}$ gives, for every integer $K\geqslant0$, every $\lambda\in\Rbb$ and every $z\in\Cbb_+$, we obtain
    \[
    \frac{1}{\lambda-z}=-\sum_{k=0}^{K}\frac{\lambda^k}{z^{k+1}}+\frac{1}{z^{K+1}}\cdot\frac{\lambda^{K+1}}{\lambda-z},
    \]
    so that, applying $\Eds\,\tr$ to $H_q$ and using $\Vert (H_q-z)^{-1}\Vert\leqslant(\Im z)^{-1}$,
    \begin{equation}
    \label{eq:asympexp}
    \left\vert m_N(z)+\sum_{k=0}^{K}\frac{\mu_k}{z^{k+1}}\right\vert
    \leqslant
    \frac{\mu^\ast_{K+1}}{\vert z\vert^{K+1}\,\Im z}
    \qquad\text{for every }z\in\Cbb_+.
    \end{equation}
    Every point $w$ of the disc of radius $\frac{\Im z}{2}$ centred at $z$ satisfies $\Im w\geqslant\frac{\Im z}{2}$ and $\vert w\vert\geqslant\frac{\vert z\vert}{2}$, so that Cauchy's integral formula applied to the difference in \eqref{eq:asympexp} shows that the expansion may be differentiated: for all integers $K,j\geqslant0$ there is $C_{K,j}>0$, depending also on $N$ and $q$, such that
    \begin{equation}
    \label{eq:asympexpderiv}
    \left\vert m_N^{(j)}(\i\eta)+\sum_{k=0}^{K}\frac{(-1)^j(k+j)!}{k!}\cdot\frac{\mu_k}{(\i\eta)^{k+j+1}}\right\vert
    \leqslant
    \frac{C_{K,j}}{\eta^{K+2+j}}
    \qquad\text{for }\eta\geqslant1.
    \end{equation}
    
    To identify the coefficients $\gamma_r$, it suffices to compare the
asymptotic expansions along the imaginary axis as $\eta\to\infty$.
We therefore set $z=\i\eta$, so that $\vert z\vert=\eta$, and apply
\eqref{eq:asympexpderiv} with $K=2R+2$. Since $\mu_0=1$ and the odd
moments vanish, we obtain
\[
    m_N'(z)+zm_N(z)+1
    =\sum_{n=0}^{R}
    \bigl[(2n+1)\mu_{2n}-\mu_{2n+2}\bigr]z^{-2n-2}
    +\O{\eta^{-2R-3}},
\]
and, for every $1\leqslant r\leqslant R$,
\[
    m_N^{(2r+1)}(z)
    =\sum_{k=0}^{R-r}
    \frac{(2k+2r+1)!}{(2k)!}\,
    \mu_{2k}\,z^{-2k-2r-2}
    +\O{\eta^{-2R-4}}.
\]
Here we retain only the powers $z^{-2n-2}$ with $n\leqslant R$;
all remaining terms in the finite expansions are $\O{\eta^{-2R-4}}$.
The larger remainder in the first identity comes from multiplying the
$\O{\eta^{-2R-4}}$ remainder of $m_N(z)$ by $z$.
Substituting these expansions into the assumed differential identity yields
    \[
    \Escr(\i\eta)=\sum_{n=0}^{R}\frac{\Lambda_n}{(\i\eta)^{2n+2}}+\O{\eta^{-2R-3}},
    \qquad
    \Lambda_n=(2n+1)\mu_{2n}-\mu_{2n+2}-\sum_{r=1}^{\min(n,R)}\frac{(2n+1)!\,\gamma_r\alpha^r}{(2r+1)!\,(2n-2r)!}\,\mu_{2n-2r},
    \]
    the implicit constant depending on $N,q,R$ and on $\gamma_1,\dots,\gamma_R$, but not on $\eta$.

    If $\Lambda_n\neq0$ for some $n\leqslant R$, take $n$ minimal with this property. Then $\Escr(\i\eta)=\Lambda_n(\i\eta)^{-2n-2}+\O{\eta^{-2n-4}}+\O{\eta^{-2R-3}}$, and both error terms are $o(\eta^{-2n-2})$ since $n\leqslant R$; letting $\eta\to\infty$ would therefore give $\vert \Escr(\i\eta)\vert \geqslant \frac12\vert \Lambda_n\vert \eta^{-2n-2}$ for $\eta$ large, contradicting the assumed bound since $2n+2\leqslant 2R+2<2R+4$. Therefore $\Lambda_n=0$ for every $n\leqslant R$, and comparing with \eqref{eq:crdef} gives $\gamma_r=c_r$ by induction on $r$, the term $r=n$ of the sum defining $\Lambda_n$ contributing exactly $\gamma_n\alpha^n$.
\end{proof}

Note that the hypothesis of Lemma~\ref{lem:coeffmoments} is satisfied by
 the error term of \eqref{eq:highode}, since
$\Im m_{N,\tau}(E+\i\eta)\leqslant\eta^{-1}$.
We now prove stability of the differential equation \eqref{eq:highode}.
We include all positive spectral heights in the statement so that the
same induction controls both the derivatives and the vertical integral
appearing in the proof.

\begin{proposition}
\label{pro:higherstab}
Let $R\geqslant0$ be an integer, and assume the conditions on $N$ and $q$
of Theorem~\ref{theo:locallaw}.
There exist constants $c_R^\ast,C_R>0$, depending only on $R$, such that,
for $N$ sufficiently large, every $\tau\in\{\text{Full},+,-\}$ and every
$z=E+\i\eta$ with $E\in\Rbb$, $\eta>0$ and
$\alpha_{N,q}\leqslant c_R^\ast\eta^2$, we have
\begin{equation}
\label{eq:higherstab}
    \left\vert m_{N,\tau}(z)-\Phi_R(z)\right\vert
    \leqslant
    \frac{C_R\alpha_{N,q}^{R+1}}{\eta^{2R+2}(1+\eta)}.
\end{equation}
\end{proposition}
\begin{proof}
We omit the $\tau$-dependence and write $\alpha=\alpha_{N,q}$.
We argue by induction on $R$. The case $R=0$ follows from
Lemma~\ref{lem:expect} for $\eta\leqslant1$ and
Lemma~\ref{lem:largeeta} for $\eta\geqslant1$, after adjusting the constants.
Assume the statement at order $R$, and set
$d=m_N-\Phi_{R+1}$.
We choose $c_{R+1}^\ast\leqslant\min(c_R^\ast/4,1)$ sufficiently small
that Lemma~\ref{lem:apriori} also applies whenever $\eta\leqslant1$.

Since $\Phi_{R+1}=\Phi_R+\alpha^{R+1}\Theta_{R+1}$, the induction
hypothesis and \eqref{eq:thetabounds} give
\begin{equation}
\label{eq:dcrude}
    |d(w)|\leqslant
    \frac{C_R\alpha^{R+1}}{(\Im w)^{2R+2}(1+\Im w)}
    \qquad\text{for }\Im w\geqslant\eta/2.
\end{equation}
Indeed, the induction hypothesis applies at these heights because
$\alpha\leqslant(c_R^\ast/4)\eta^2\leqslant c_R^\ast(\Im w)^2$;
the added profile is uniformly bounded at heights at most one and
decays as $(\Im w)^{-2R-5}$ at larger heights.
Cauchy's integral formula on the disc of radius $s/2$ centred at $E+\i s$
therefore yields, for $s\geqslant\eta$ and $1\leqslant r\leqslant R+1$,
\[
    |d^{(2r+1)}(E+\i s)|
    \leqslant
    \frac{C_R\alpha^{R+1}}{s^{2R+2r+3}(1+s)}.
\]
In particular, since $\alpha/s^2\leqslant1$,
\begin{equation}
\label{eq:dderivsource}
    \sum_{r=1}^{R+1}\frac{|c_r|\alpha^r}{(2r+1)!}
    |d^{(2r+1)}(E+\i s)|
    \leqslant
    \frac{C_R\alpha^{R+2}}{s^{2R+5}(1+s)}.
\end{equation}

We next write the equation satisfied by $d$.
Summing the recursion \eqref{eq:thetarec} and using $L[m]=-1$, we obtain
\[
    L[\Phi_{R+1}]+1
    =\sum_{r=1}^{R+1}\frac{c_r\alpha^r}{(2r+1)!}
    \Phi_{R+1}^{(2r+1)}-\Ycal_{R+1},
\]
where the terms beyond the truncation order are
\[
    \Ycal_{R+1}
    \coloneqq
    \sum_{\substack{1\leqslant r\leqslant R+1,\ 0\leqslant t\leqslant R+1\\
                    r+t\geqslant R+2}}
    \frac{c_r\alpha^{r+t}}{(2r+1)!}\Theta_t^{(2r+1)}.
\]
Subtracting this identity from \eqref{eq:highode} at order $R+1$ gives
\begin{equation}
\label{eq:dRode}
    d'(z)+zd(z)
    =\sum_{r=1}^{R+1}\frac{c_r\alpha^r}{(2r+1)!}d^{(2r+1)}(z)
    +\Ycal_{R+1}(z)+\Escr_{R+1}(z).
\end{equation}
Every summand of $\Ycal_{R+1}$ has $t\geqslant1$ and $r+t\geqslant R+2$.
By \eqref{eq:thetabounds}, its derivative factor is uniformly bounded
for $s\leqslant1$ and bounded by $C_Rs^{-2t-2r-4}$ for $s\geqslant1$.
Since $\alpha\leqslant1$, this gives
\[
    |\Ycal_{R+1}(E+\i s)|
    \leqslant\frac{C_R\alpha^{R+2}}{(1+s)^{2R+6}}
    \leqslant\frac{C_R\alpha^{R+2}}{s^{2R+5}(1+s)}.
\]
Moreover, Lemma~\ref{lem:apriori} for $s\leqslant1$ and the bound
$\Im m_N(E+\i s)\leqslant s^{-1}$ for $s\geqslant1$ imply
\[
    \Im m_N(E+\i s)\leqslant\frac{C}{1+s}
    \qquad\text{for }s\geqslant\eta.
\]
Thus \eqref{eq:highode} gives
\[
    |\Escr_{R+1}(E+\i s)|
    \leqslant\frac{C_R\alpha^{R+2}}{s^{2R+5}(1+s)}.
\]
Combining these bounds with \eqref{eq:dderivsource} controls the
right-hand side of \eqref{eq:dRode} by the same expression.
Finally, $d(E+\i s)\to0$ as $s\to\infty$, so integration along the
vertical line, as in \eqref{eq:bounddiff}, yields
\[
    |d(E+\i\eta)|
    \leqslant C_R\alpha^{R+2}
    \int_\eta^\infty
    \frac{\e^{(\eta^2-s^2)/2}}{s^{2R+5}(1+s)}\d s\leqslant\frac{C_R\alpha^{R+2}}{1+\eta}
    \int_\eta^\infty s^{-2R-5}\d s
    \leqslant\frac{C_R\alpha^{R+2}}{\eta^{2R+4}(1+\eta)}.
\]
This proves \eqref{eq:higherstab} at order $R+1$.
\end{proof}

We now prove concentration of the Stieltjes transform through the following lemma
\begin{lemma}\label{lem:concent}
    There exist absolute constants $c,C>0$ such that, for every $E\in\Rbb$, $\eta\in(0,1]$, $t>0$, and $\tau\in\{\text{Full},+,-\}$,
    \begin{equation}
    \label{eq:concent}
    \Pds\left(
        \vert s_{N,\tau}(E+\i\eta)-m_{N,\tau}(E+\i\eta)\vert \geqslant t
    \right)\leqslant C\exp\left(-c\,M\eta^{3}\min\left(\frac{t^2}{\Im m_{N,\tau}(E+\i\eta)},\,t\right)\right).
    \end{equation}
    In particular, if $\Im m_{N,\tau}(E+\i\eta)\leqslant C_0$ for some $C_0\geqslant1$, then for every $0<t\leqslant1$
    \begin{equation}
    \label{eq:concentsimple}
    \Pds\left(
        \vert s_{N,\tau}(E+\i\eta)-m_{N,\tau}(E+\i\eta)\vert \geqslant t
    \right)\leqslant C\exp\left(-\frac{c}{C_0}\,M\eta^{3}t^2\right).
    \end{equation}
\end{lemma}
\begin{proof}
    We omit the $\tau$-dependence and write $z=E+\i\eta$. We first note that we have
    \[
    \partial_{J_A}s_N(z) = -\frac{1}{\sqrt{M}}\tr(G(z)\Gamma_AG(z)) = -\frac{1}{\sqrt{M}}\tr(\Gamma_AG(z)^2).
    \]
    Using orthonormality of the family $(\Gamma_A)_A$ for the normalized Hilbert--Schmidt inner product, Bessel's inequality  and the Ward identity \eqref{eq:boundG4} give
    \begin{equation}
    \label{eq:gradsN}
    \sum_{\substack{A\subset\unn{1}{N}\\ \vert A\vert=q}}
    \vert \partial_{J_A}s_N(z)\vert^2 \leqslant \frac{1}{M}\sum_{\substack{A\subset\unn{1}{N}\\\vert A\vert=q}}\left\vert \tr(\Gamma_AG(z)^2)\right\vert^2\leqslant\frac{1}{M}\tr(\vert G(z)\vert^4)\leqslant \frac{\Im s_N(z)}{M\eta^3}.
    \end{equation}
    %\emph{Step 1: concentration of $\Im s_N(z)$.} 
    The function $J\mapsto\Im s_N(z)$ is smooth and positive, and $\vert\partial_{J_A}\Im s_N(z)\vert\leqslant\vert\partial_{J_A}s_N(z)\vert$, so that \eqref{eq:gradsN} gives
    \[
    \left\vert\nabla\sqrt{\Im s_N(z)}\right\vert^2=\frac{\vert\nabla\Im s_N(z)\vert^2}{4\,\Im s_N(z)}\leqslant\frac{1}{4M\eta^3}.
    \]
    Hence $\sqrt{\Im s_N(z)}$ is a Lipschitz function of the Gaussian vector $(J_A)_A$, with constant $(2\sqrt M\eta^{\frac32})^{-1}$, and Gaussian concentration together with Jensen's inequality $\Eds\sqrt{\Im s_N(z)}\leqslant\sqrt{\Im m_N(z)}$ gives
    \[
    \Pds\left(\sqrt{\Im s_N(z)}\geqslant\sqrt{\Im m_N(z)}+u\right)\leqslant\exp\left(-2M\eta^3u^2\right),\qquad u>0.
    \]
    Integrating this tail bound, we obtain for every $p\geqslant1$
    \begin{equation}
    \label{eq:sqrtmoments}
    \left\Vert\sqrt{\Im s_N(z)}\right\Vert_p\leqslant\sqrt{\Im m_N(z)}+\left\Vert\Big(\sqrt{\Im s_N(z)}-\sqrt{\Im m_N(z)}\Big)_+\right\Vert_p\leqslant\sqrt{\Im m_N(z)}+\frac{C\sqrt p}{\sqrt M\eta^{\frac32}},
    \end{equation}
    where $\Vert X\Vert_p\coloneqq(\Eds\vert X\vert^p)^{\frac1p}$.
    Applying the Gaussian $L^p$ Poincar\'e inequality to $\Re s_N(z)$ and $\Im s_N(z)$ separately, for every $p\geqslant2$ we get, by \eqref{eq:gradsN} and \eqref{eq:sqrtmoments},
    \begin{equation}
    \label{eq:sNmoments}
    \left\Vert s_N(z)-m_N(z)\right\Vert_p
    \leqslant C\sqrt p\,\Bigg\Vert\Big(\sum_{\vert A\vert=q}\vert\partial_{J_A}s_N(z)\vert^2\Big)^{\frac12}\Bigg\Vert_p
    \leqslant\frac{C\sqrt p}{\sqrt M\eta^{\frac32}}\left\Vert\sqrt{\Im s_N(z)}\right\Vert_p
    \leqslant\frac{C\sqrt{p\,\Im m_N(z)}}{\sqrt M\eta^{\frac32}}+\frac{Cp}{M\eta^3}.
    \end{equation}

    Fix $t>0$ and set $p\coloneqq c\,M\eta^3\min\big(t^2/\Im m_N(z),t\big)$ with $c>0$ small enough. If $p<2$ the claim \eqref{eq:concent} is trivial provided $C$ is large enough. If $p\geqslant2$, the two terms in the right-hand side of \eqref{eq:sNmoments} are respectively at most $C\sqrt c\,t$ and $Cc\,t$, and thus both at most $\frac{t}{2\e}$ for $c$ small enough, so that Markov's inequality gives
    \[
    \Pds\left(\vert s_N(z)-m_N(z)\vert\geqslant t\right)\leqslant t^{-p}\left\Vert s_N(z)-m_N(z)\right\Vert_p^p\leqslant\e^{-p},
    \]
    which is \eqref{eq:concent}. Finally, if $\Im m_N(z)\leqslant C_0$ with $C_0\geqslant1$ and $t\leqslant1$, then $\min(t^2/\Im m_N(z),t)\geqslant\min(t^2/C_0,t)=t^2/C_0$, which gives \eqref{eq:concentsimple}.
\end{proof}
We are now ready to prove the local law.

\begin{proof}[Proof of Theorem~\ref{theo:locallaw}]
We omit the $\tau$-dependence. Since
\[
    \frac{\alpha_{N,q}}{\eta^2}
    \leqslant(\log N)^{-2\varepsilon}
    \qquad\text{on }\Dcal_{L,\varepsilon},
\]
Proposition~\ref{pro:higherstab} and Lemma~\ref{lem:apriori} give,
for $N$ sufficiently large,
\[
    |m_N(z)-\Phi_R(z)|
    \leqslant\frac{C_R\alpha_{N,q}^{R+1}}{\eta^{2R+2}},
    \qquad
    \Im m_N(z)\leqslant C.
\]
Combining these estimates with \eqref{eq:concentsimple} and the triangle
inequality proves \eqref{eq:llawave}.

To obtain the uniform estimate, let $\Pcal_N$ be an $M^{-1}$-net of
$\Dcal_{L,\varepsilon}$ with $|\Pcal_N|\leqslant C_LM^2$.
Apply \eqref{eq:concentsimple} at each $z=E+\i\eta\in\Pcal_N$ with
\[
    t(z)=A\frac{\sqrt{\log M}}{\sqrt M\,\eta^{3/2}},
\]
where $A$ is a sufficiently large absolute constant.
This choice satisfies $t(z)\leqslant1$ throughout the domain for $N$
sufficiently large. Indeed, writing
$\eta_*=(\log N)^\varepsilon\alpha_{N,q}^{1/2}$, we have
\[
    \frac{M\eta_*^3}{\log M}\longrightarrow\infty,
\]
using $\alpha_{N,q}\geqslant cN^{-1}$ and $M\geqslant\binom N2$.
A union bound therefore shows that, with probability at least $1-M^{-10}$,
\[
    |s_N(z)-m_N(z)|
    \leqslant A\frac{\sqrt{\log M}}{\sqrt M\,\eta^{3/2}}
    \qquad\text{for every }z\in\Pcal_N.
\]

To extend this estimate to the whole domain, use
\[
    |s_N'(z)|+|m_N'(z)|\leqslant2\eta^{-2}.
\]
For any $z=E+\i\eta\in\Dcal_{L,\varepsilon}$, choose $z_0\in\Pcal_N$
with $|z-z_0|\leqslant M^{-1}$.
Since $M^{-1}=o(\eta_*)$, the imaginary parts of $z$ and $z_0$ are
comparable, and
\[
    |(s_N-m_N)(z)-(s_N-m_N)(z_0)|
    \leqslant\frac{C}{M\eta^2}
    \leqslant C\frac{\sqrt{\log M}}{\sqrt M\,\eta^{3/2}}.
\]
Thus, on the same event,
\[
    |s_N(z)-m_N(z)|
    \leqslant C\frac{\sqrt{\log M}}{\sqrt M\,\eta^{3/2}}
    \qquad\text{throughout }\Dcal_{L,\varepsilon}.
\]
Adding the deterministic bound on $m_N-\Phi_R$ proves
\eqref{eq:llawuniform}.
\end{proof}

We now aim to prove the asymptotic of the spectral form factor in Corollary~\ref{cor:sff}. We start by using our expected local law to control the Fourier transform of the expected empirical spectral measure.

\begin{lemma}
\label{lem:sff1}
Let $\varepsilon_0>0$, let $q$ be even with $2\leqslant q\leqslant N^{\frac{1}{2}-\varepsilon_0}$, and let $R\geqslant0$ be an integer.
For $\tau\in\{\text{Full},+,-\}$, define
\[
\phi_{N,\tau}(t)
=
\Eds\left[
\tr_\tau(\e^{-\i tH_q})
\right].
\]
There exist constants $c_R,C_R>0$ such that, if $T\geqslant1$ satisfies $\alpha_{N,q}T^2\leqslant c_R$, then
\begin{equation}
\label{eq:sffphi}
\sup_{|t|\leqslant T}
\left|
\phi_{N,\tau}(t)-\widehat{\varrho_R}(t)
\right|
\leqslant
C_R\left(
\alpha_{N,q}^{R+1}T^{2R+2}
\right)^{\frac{4}{5}},
\end{equation}
with $\widehat{\varrho_R}$ as in \eqref{eq:hatvarrho}.
\end{lemma}

\begin{proof}
For $\tau\in\{\text{Full},+,-\}$, define
\[
\mu_{N,\tau}
=
\frac{1}{D^\tau}
\sum_{j=1}^{D^\tau}\delta_{\lambda_j^\tau},
\qquad
\overline{\mu}_{N,\tau}
=
\Eds[\mu_{N,\tau}].
\]
We suppress the dependence on $\tau$ in the remainder of the proof.
For $\eta\in(0,1]$ set
\[
\rho_{N,\eta}=P_\eta\ast\overline{\mu}_N,
\qquad\quad
\varrho_{R,\eta}=P_\eta\ast\varrho_R,
\]
where $P_\eta$ is the Poisson kernel and $\varrho_R$ is the profile density of \eqref{eq:defvarrhoR}. By Stieltjes inversion and Lemma~\ref{lem:thetastruct},
\[
\rho_{N,\eta}(E)
=
\frac{1}{\pi}\Im m_N(E+\i\eta),
\qquad\quad
\varrho_{R,\eta}(E)
=
\frac{1}{\pi}\Im \Phi_R(E+\i\eta).
\]
Under the Fourier convention
\[
\widehat{f}(t)
=
\int_\Rbb \e^{-\i tE}f(E)\d E
\quad\text{we have}\quad
\widehat{P}_\eta(t)=\e^{-\eta|t|},
\quad
\widehat{\rho}_{N,\eta}(t)
=
\e^{-\eta|t|}\phi_N(t),
\quad\text{and}\quad
\widehat{\varrho}_{R,\eta}(t)
=
\e^{-\eta|t|}\widehat{\varrho_R}(t),
\]
so that
\begin{equation}
\label{eq:fourierlone}
\left|
\phi_N(t)-\widehat{\varrho_R}(t)
\right|
\leqslant
\e^{\eta|t|}
\Vert\rho_{N,\eta}-\varrho_{R,\eta}\Vert_1.
\end{equation}

We first control this $L^1$-norm on a bounded interval. Provided $\alpha_{N,q}\leqslant c^\ast_R\eta^2$, Proposition~\ref{pro:higherstab} gives
\[
\left|
\rho_{N,\eta}(E)-\varrho_{R,\eta}(E)
\right|
=
\frac{1}{\pi}\left|\Im\left(m_N(E+\i\eta)-\Phi_R(E+\i\eta)\right)\right|
\leqslant
A_\eta
\coloneqq
\frac{C_R\,\alpha_{N,q}^{R+1}}{\eta^{2R+2}},
\]
uniformly in $E\in\Rbb$. Consequently, for every $K\geqslant2$,
\begin{equation}
\label{eq:sffbulk}
\int_{-K}^K
\left|
\rho_{N,\eta}(E)-\varrho_{R,\eta}(E)
\right|
\d E
\leqslant
2KA_\eta.
\end{equation}
We next obtain an estimate on the complement of $[-K,K]$ by exploiting moment cancellation. Set $\nu_N=\overline{\mu}_N-\varrho_R\,\d x$. Since $\overline\mu_N$ is a probability measure and $\varrho_R$ has total mass one by Lemma~\ref{lem:thetastruct}, we have $\nu_N(\Rbb)=0$.
The law of $H_q$ is invariant under $H_q\mapsto-H_q$, and hence the odd moments of $\overline{\mu}_N$ vanish. Moreover,
\[
\int_\Rbb x^2\overline{\mu}_N(\d x)
=
\Eds\left[\tr(H_q^2)\right]
=
1.
\]
On the other hand $\varrho_R-\rho_G$ is, by \eqref{eq:defvarrhoR} and Lemma~\ref{lem:thetastruct}, a linear combination of derivatives $\rho_G^{(k)}$ with $k$ even and $k\geqslant4$, and $\int_\Rbb x^j\rho_G^{(k)}(x)\d x=0$ whenever $0\leqslant j<k$. Since $\rho_G$ is even with unit mass and unit variance, we conclude that
\begin{equation}
\label{eq:momentcancel}
\int_\Rbb x^k\nu_N(\d x)=0,
\qquad
k=0,1,2,3.
\end{equation}
We also need a uniform fourth-moment bound. Wick's formula gives
\[
\begin{aligned}
\int_\Rbb x^4\overline{\mu}_N(\d x)
&=
\Eds\left[\tr(H_q^4)\right]=
2+
\frac{1}{M^2}
\sum_{\substack{|A|=q\\|B|=q}}
\tr\left(
\Gamma_A\Gamma_B\Gamma_A\Gamma_B
\right).
\end{aligned}
\]
For fixed $A$, the identity
\[
\Gamma_A\Gamma_B\Gamma_A\Gamma_B
=
\begin{cases}
\Id,&\Gamma_A\Gamma_B=\Gamma_B\Gamma_A,\\
-\Id,&\Gamma_A\Gamma_B=-\Gamma_B\Gamma_A
\end{cases}
\]
and the definition of $\alpha_{N,q}$ give
\[
\frac{1}{M}
\sum_{|B|=q}
\tr\left(
\Gamma_A\Gamma_B\Gamma_A\Gamma_B
\right)
=
1-2\alpha_{N,q},
\]
so that
\begin{equation}
\label{eq:fourthmomentsyk}
\int_\Rbb x^4\overline{\mu}_N(\d x)
=
3-2\alpha_{N,q}
\leqslant3.
\end{equation}
Since $\int_\Rbb x^4\rho_G(x)\d x=3$ and $\int_\Rbb|x|^4\left\vert\varrho_R(x)-\rho_G(x)\right\vert\d x\leqslant C_R\alpha_{N,q}$ by \eqref{eq:defvarrhoR}, it follows that
\begin{equation}
\label{eq:fourthvariation}
\int_\Rbb |x|^4|\nu_N|(\d x)\leqslant C_R.
\end{equation}
We now estimate the tails of
$\rho_{N,\eta}-\varrho_{R,\eta}
=
P_\eta\ast\nu_N.$
For $E,x\in\Rbb$, define the third-order Taylor remainder
\[
\mathcal{E}_\eta(E,x)
=
P_\eta(E-x)
-
\sum_{k=0}^3
\frac{(-x)^k}{k!}P_\eta^{(k)}(E).
\]
By \eqref{eq:momentcancel}, we have
\begin{equation}
\label{eq:poissontaylor}
\rho_{N,\eta}(E)-\varrho_{R,\eta}(E)
=
\int_\Rbb
\mathcal{E}_\eta(E,x)\nu_N(\d x).
\end{equation}
For $0\leqslant k\leqslant4$, the derivatives of the Poisson kernel satisfy
\begin{equation}
\label{eq:poissonderivative}
\left|
P_\eta^{(k)}(E)
\right|
\leqslant
\frac{C_k\eta}{|E|^{k+2}},
\qquad
E\neq0.
\end{equation}
Indeed $P_\eta(E)=\eta^{-1}P_1(E/\eta)$, so that $P_\eta^{(k)}(E)=\eta^{-1-k}P_1^{(k)}(E/\eta)$, while $\vert P_1^{(k)}(v)\vert\leqslant C_k(1+\vert v\vert)^{-k-2}$.
Suppose first that $|x|\leqslant K/2$ and $|E|>K$. Taylor's formula and \eqref{eq:poissonderivative} give
\[
\left|
\mathcal{E}_\eta(E,x)
\right|
\leqslant
C|x|^4
\sup_{0\leqslant s\leqslant1}
\left|
P_\eta^{(4)}(E-sx)
\right|
\leqslant
\frac{C\eta|x|^4}{|E|^6}.
\]
Therefore,
\begin{equation}
\label{eq:smallxtail}
\int_{|E|>K}
\left|
\mathcal{E}_\eta(E,x)
\right|
\d E
\leqslant
\frac{C\eta|x|^4}{K^5}
\leqslant
\frac{C|x|^4}{K^4},
\end{equation}
where we used $\eta\leqslant1\leqslant K$. For $|x|>K/2$, we use directly the definition of
$\mathcal{E}_\eta$. Since $P_\eta$ is a probability density and
\eqref{eq:poissonderivative} gives
\[
\int_{|E|>K}
\left|
P_\eta^{(k)}(E)
\right|
\d E
\leqslant
\frac{C_k\eta}{K^{k+1}},
\]
we obtain
\[
\int_{|E|>K}
\left|
\mathcal{E}_\eta(E,x)
\right|
\d E
\leqslant
1+
C\sum_{k=0}^3
\frac{\eta|x|^k}{K^{k+1}}.
\]
Using \eqref{eq:fourthvariation}, we have, for $0\leqslant k\leqslant3$,
\[
\int_{|x|>K/2}
|x|^k|\nu_N|(\d x)
\leqslant
\left(\frac{2}{K}\right)^{4-k}
\int_\Rbb |x|^4|\nu_N|(\d x)
\leqslant
\frac{C_R}{K^{4-k}}.
\]
It follows that
\begin{equation}
\label{eq:largextail}
\int_{|x|>K/2}
\int_{|E|>K}
\left|
\mathcal{E}_\eta(E,x)
\right|
\d E\,|\nu_N|(\d x)
\leqslant
\frac{C_R}{K^4}.
\end{equation}
Combining \eqref{eq:poissontaylor}, \eqref{eq:smallxtail},
\eqref{eq:largextail}, and \eqref{eq:fourthvariation}, we conclude that
\begin{equation}
\label{eq:sfftailfourth}
\int_{|E|>K}
\left|
\rho_{N,\eta}(E)-\varrho_{R,\eta}(E)
\right|
\d E
\leqslant
\frac{C_R}{K^4},
\end{equation}
and therefore, together with \eqref{eq:sffbulk},
\begin{equation}
\label{eq:sfflonefourth}
\Vert\rho_{N,\eta}-\varrho_{R,\eta}\Vert_1
\leqslant
C_R\left(
KA_\eta+\frac{1}{K^4}
\right).
\end{equation}

We now take $\eta=\frac{1}{T}$, so that
\[
A_{\frac1T}=C_R\,\alpha_{N,q}^{R+1}T^{2R+2}=C_R\left(\alpha_{N,q}T^2\right)^{R+1}.
\]
Under our assumption $\alpha_{N,q}T^2\leqslant c_R$, after decreasing $c_R$ if necessary, we have both $\alpha_{N,q}\leqslant c^\ast_R\eta^2$, as required above, and $A_{\frac1T}\leqslant2^{-5}$, so that we may choose $K=A_{\frac1T}^{-\frac15}\geqslant2$. The two terms in
\eqref{eq:sfflonefourth} are then of the same order, and hence
\[
\Vert\rho_{N,\frac{1}{T}}-\varrho_{R,\frac1T}\Vert_1
\leqslant
C_RA_{\frac1T}^{\frac{4}{5}}.
\]
Finally, for $|t|\leqslant T$ we have $\e^{\eta|t|}\leqslant\e$, and the conclusion follows from \eqref{eq:fourierlone}.
\end{proof}
\eer

We are now ready to prove Corollary~\ref{cor:sff}.

\begin{proof}[Proof of Corollary~\ref{cor:sff}]
We omit the dependence on $\tau$. We can write
\begin{equation}
\label{eq:sffdecomposition}
K_N(t)
=
|\phi_N(t)|^2
+
\Var\left(
\tr(\e^{-\i tH_q})
\right).
\end{equation}
We first estimate the variance term. Let
$A\subset\unn{1}{N}$ be a $q$-subset. Duhamel's formula and cyclicity of the trace give
\[
\begin{aligned}
\partial_{J_A}
\tr(\e^{-\i tH_q})
&=
-\frac{\i}{\sqrt{M}}
\int_0^t
\tr\left(
\e^{-\i(t-s)H_q}
\Gamma_A
\e^{-\i sH_q}
\right)
\d s
=
-\frac{\i t}{\sqrt{M}}
\tr\left(
\Gamma_A\e^{-\i tH_q}
\right).
\end{aligned}
\]
By orthonormality of the family $(\Gamma_A)_A$ for the normalized Hilbert--Schmidt inner product and Bessel's inequality,
\[
\sum_{\substack{A\subset\unn{1}{N}\\|A|=q}}
\left|
\tr\left(
\Gamma_A\e^{-\i tH_q}
\right)
\right|^2
\leqslant
\tr\left(
(\e^{-\i tH_q})^*
\e^{-\i tH_q}
\right)
=
1.
\]
The Gaussian Poincar\'e inequality therefore gives
\begin{equation}
\label{eq:sffvariance}
\Var\left(
\tr(\e^{-\i tH_q})
\right)
\leqslant
\Eds\left[
\sum_{\substack{A\subset\unn{1}{N}\\|A|=q}}
\left|
\partial_{J_A}
\tr(\e^{-\i tH_q})
\right|^2
\right]
\leqslant
\frac{t^2}{M}.
\end{equation}
Since $|\phi_N(t)|\leqslant1$ and $\left\vert\widehat{\varrho_R}(t)\right\vert\leqslant\int_\Rbb\vert\varrho_R\vert\leqslant C_R$, we have
\begin{equation}
\label{eq:sffsquaremean}
\left|
|\phi_N(t)|^2-\left\vert\widehat{\varrho_R}(t)\right\vert^2
\right|
\leqslant
C_R\left|
\phi_N(t)-\widehat{\varrho_R}(t)
\right|.
\end{equation}
Moreover, by \eqref{eq:sfftimerange} we have $T^2\leqslant\frac{4(R+1)}{5}\log\frac{1}{\alpha_{N,q}}$ and therefore
\[
\alpha_{N,q}T^2\leqslant \frac{4(R+1)}{5}\,\alpha_{N,q}\log\frac{1}{\alpha_{N,q}}\leqslant c_R
\]
for $N$ large enough, so that Lemma~\ref{lem:sff1} applies. Combining it with
\eqref{eq:sffdecomposition}, \eqref{eq:sffvariance}, and
\eqref{eq:sffsquaremean}, we obtain
\[
\sup_{|t|\leqslant T}
\left|
K_N(t)-\left\vert\widehat{\varrho_R}(t)\right\vert^2
\right|
\leqslant
C_R\left[
\left(
\alpha_{N,q}^{R+1}T^{2R+2}
\right)^{\frac{4}{5}}
+
\frac{T^2}{M}
\right],
\]
which is the first inequality in \eqref{eq:boundsff}. The second one follows from $T^2\leqslant\frac{4(R+1)}{5}\log\frac{1}{\alpha_{N,q}}$ and $T^2\leqslant\log M$, which give
\[
\alpha_{N,q}^{R+1}T^{2R+2}\leqslant C_R\left(\alpha_{N,q}\log\frac{1}{\alpha_{N,q}}\right)^{R+1}
\qquad\text{and}\qquad
\frac{T^2}{M}\leqslant\frac{\log M}{M}.
\]
\end{proof}

\section{Proof of weak delocalization}
In this section we prove  Theorem~\ref{theo:isolaw} and Corollary~\ref{theo:delocinfinity}. A key input for the proof of this theorem is the following deterministic bound, whose proof is presented after the proof of Theorem~\ref{theo:isolaw} is concluded.

 \begin{proposition}
    \label{pro:maincliffpro}
    Let $q$ be even and fixed. We define the Fock space 
    \[
    \Fcal_{n} \coloneqq \left\{ \psi = (\psi_S)_{S\subset\unn{1}{n}}, \sum_{S\subset\unn{1}{n}}\vert \psi_S\vert^2<\infty\right\} \cong \Cbb^{D}.
    \]
    There exists $C_q>0$\footnote{The dependence of $C_q$ on $q$ is not explicit in \cite{visconti2026}, which is why we restrict this result to fixed $q$.} such that for $v,w\in\Fcal_n,$ we have
    \[
    \frac{1}{M}\sum_{\substack{A\subset\unn{1}{N}\\\vert A\vert=q}} \vert \langle v,\Gamma_Aw\rangle\vert^2 
    \leqslant C_q N^{-\frac{q}{2}}\Vert v\Vert^2\Vert w\Vert^2.
    \]
    The same bound holds after restriction to either parity sector.
    \end{proposition}
    
In order to prove Theorem~\ref{theo:isolaw}, we start with the following lemma showing that in expectation the diagonal resolvent entries are exactly equal to its averaged trace.
 \begin{lemma}
    \label{lem:exptrick}
    For any $\tau\in \{+,-\}$, and any $z\in\mathbb{C}_+$, we have 
    \begin{equation}
    \E[G_\tau(z)]=m_{N,\tau}(z)I_{\mathcal{H}_N^\tau}
    \end{equation}
    \end{lemma}
\begin{proof}
Consider $B\subset \unn{1}{N}$ with $|B|$ even. Consider any even Clifford unitary monomial $U_B$ chosen to be a phase multiple of $\prod_{j\in B}\gamma_j$.

For every $q$-subset $A$ we obtain
\[
U_B\Gamma_AU_B^*=\epsilon_{A,B}\Gamma_A, \qquad\quad \epsilon_{A,B}=\pm 1.
\]
We thus have
\[
U_BH_q(J)U_B^*=H_q(\epsilon J),
\]
where $(\epsilon J)_A:=\epsilon_{A,B} J_A$. Since the random vectors $J$ and $\epsilon J$ have the same law, we obtain

\[
U_B\E G_\tau(z)U_B^*=\E G_\tau(z).
\]
This shows that $\E G_\tau(z)$ commutes with all even Clifford monomials and so it must be a scalar multiple of the identity.
\end{proof}

Next, fix ${\bm x},{\bm y}\in\Hcal_N^\tau$  with
$\Vert{\bm x}\Vert=\Vert{\bm y}\Vert=1$, and define
\begin{equation}
\label{eq:defF}
F_{{\bm x},{\bm y}}^{(E)}(J)
\coloneqq
\left\langle {\bm x},G_\tau(E+\i\eta){\bm y}\right\rangle.
\end{equation}
For a deterministic unit vector ${\bm x}\in\Hcal_N^\tau$, we also set
\[
Y_{\bm x}^{(E)}(J)
\coloneqq
\Im F_{{\bm x},{\bm x}}^{(E)}(J)
=
\Im\left\langle{\bm x},G_\tau(E+\i\eta){\bm x}\right\rangle>0.
\]

The following lemma gives us gradient estimates on $Y_{\bm{x}}^{(E)}$ and $F_{\bm{x},\bm{y}}^{(E)}$.

\begin{lemma}
\label{lem:wardgradient}
For every realization of the disorder and every pair of deterministic unit vectors ${\bm x},{\bm y}\in\Hcal_N^\tau$, we have
\begin{equation}
\label{eq:wardgradient}
\sum_{|A|=q}
\left|
\partial_{J_A}\log Y_{\bm x}^{(E)}
\right|^2
\leqslant
\frac{C_q}{N^{\frac{q}{2}}\eta^2}
\quad\text{and}\quad 
\sum_{|A|=q}
\left|
\partial_{J_A}F_{{\bm x},{\bm y}}^{(E)}
\right|^2
\leqslant
\frac{C_q}{N^{\frac{q}{2}}\eta^2}
Y_{\bm x}^{(E)}Y_{\bm y}^{(E)}.
\end{equation}
\end{lemma}

\begin{proof}
 We have for $z=E+\i\eta,$
\[
\partial_{J_A}F_{{\bm x},{\bm y}}^{(E)}
=
-\frac{1}{\sqrt{M}}
\left\langle
G_\tau^*(z){\bm x},\Gamma_AG_\tau(z){\bm y}
\right\rangle.
\]
Proposition~\ref{pro:maincliffpro} therefore gives
\[
\sum_{|A|=q}
\left|
\partial_{J_A}F_{{\bm x},{\bm y}}^{(E)}
\right|^2
\leqslant
\frac{C_q}{N^{\frac{q}{2}}}
\Vert G_\tau^*(z){\bm x}\Vert^2
\Vert G_\tau(z){\bm y}\Vert^2.
\]
From the Ward identity, we have
\[
\Vert G_\tau(z){\bm x}\Vert^2
=
\Vert G^*_\tau(z){\bm x}\Vert^2
=
\frac{1}{\eta}
\Im\left\langle{\bm x},G_\tau(z){\bm x}\right\rangle
=
\frac{Y_{\bm x}^{(E)}}{\eta}.
\]
Applying the same identity to ${\bm y}$ proves the second inequality in \eqref{eq:wardgradient}.

Moreover,
\[
\left|
\partial_{J_A}Y_{\bm x}^{(E)}
\right|
\leqslant
\frac{1}{\sqrt{M}}
\left|
\left\langle
G^*_\tau(z){\bm x},\Gamma_AG_\tau(z){\bm x}
\right\rangle
\right|.
\]
Hence, using Proposition~\ref{pro:maincliffpro} and the Ward identity once again,
\[
\begin{aligned}
\sum_{|A|=q}
\left|
\partial_{J_A}Y_{\bm x}^{(E)}
\right|^2
&\leqslant
\frac{C_q}{N^{\frac{q}{2}}}
\Vert G^*_\tau(z){\bm x}\Vert^2
\Vert G_\tau(z){\bm x}\Vert^2
=
\frac{C_q}{N^{\frac{q}{2}}\eta^2}
\left(Y_{\bm x}^{(E)}\right)^2.
\end{aligned}
\]
Dividing by $\left(Y_{\bm x}^{(E)}\right)^2$ proves the first part of 
\eqref{eq:wardgradient}.
\end{proof}

We next record the resulting moment bound.

\begin{lemma}
\label{lem:wardmoments}
There exist constants $c,C>0$, depending only on $L$, $\delta$ and $C_*$, such that, uniformly for $|E|\leqslant L$ and $N^{-\frac{1}{2}+\delta}\leqslant\eta\leqslant 1,$
we have
\begin{equation}
\label{eq:wardmomentbound}
\left(
\Eds\left[
\left(Y_{\bm x}^{(E)}\right)^p
\right]
\right)^{\frac{1}{p}}
\leqslant C
\end{equation}
for every deterministic unit vector ${\bm x}\in\Hcal_N^\tau$ and every $2\leqslant p\leqslant cN^{\frac{q}{2}}\eta^2.$
\end{lemma}

\begin{proof}
By Lemma~\ref{lem:exptrick},
\[
\Eds\left[Y_{\bm x}^{(E)}\right]
=
\Im m_{N,\tau}(E+\i\eta).
\]
Since $q$ is bounded independently of $N$, we have
$\alpha_{N,q}\leqslant CN^{-1}$. Therefore, for
$\eta\geqslant N^{-\frac{1}{2}+\delta}$,
\[
\frac{\alpha_{N,q}}{\eta^2}
\leqslant
CN^{-2\delta}.
\]
The bound \eqref{eq:apriorimn} of Lemma~\ref{lem:apriori} now yields
\[
\Eds\left[Y_{\bm x}^{(E)}\right]\leqslant C.
\]
Markov's inequality shows that a median of
$\log Y_{\bm x}^{(E)}$ is bounded from above by an absolute constant.

On the other hand, \eqref{eq:wardgradient} shows that
$\log Y_{\bm x}^{(E)}$ is
$C_qN^{-\frac{q}{4}}\eta^{-1}$-Lipschitz as a function of the Gaussian vector $(J_A)_A$. Gaussian concentration therefore gives
\begin{equation}
\label{eq:logwardconcentration}
\Pds\left(
\log Y_{\bm x}^{(E)}\geqslant \log C+s
\right)
\leqslant
2\exp\left(-cN^{\frac{q}{2}}\eta^2s^2\right),
\qquad s>0.
\end{equation}
Integrating this tail, we obtain
\[
\begin{aligned}
\Eds\left[
\left(Y_{\bm x}^{(E)}\right)^p
\right]
&\leqslant
C^p\left(
1+
2p\int_0^\infty
\exp\left(ps-cN^{\frac{q}{2}}\eta^2s^2\right)\d s
\right)
\leqslant
C^p\exp\left(\frac{Cp^2}{N^{\frac{q}{2}}\eta^2}\right).
\end{aligned}
\]
Taking the $p$-th root proves \eqref{eq:wardmomentbound} whenever
$p\leqslant cN^{\frac{q}{2}}\eta^2$.
\end{proof}

We are now ready to prove Theorem~\ref{theo:isolaw}.

\begin{proof}[Proof of Theorem~\ref{theo:isolaw}]
The Gaussian $L^p$ Poincar\'e inequality gives, for every $p\geqslant2$,
\[
\left(
\Eds\left[
\left|
F_{{\bm x},{\bm y}}^{(E)}
-
\Eds[F_{{\bm x},{\bm y}}^{(E)}]
\right|^p
\right]
\right)^{\frac{1}{p}}
\leqslant
C\sqrt{p}
\left(
\Eds\left[
\left(
\sum_{|A|=q}
\left|
\partial_{J_A}F_{{\bm x},{\bm y}}^{(E)}
\right|^2
\right)^{p/2}
\right]
\right)^{\frac{1}{p}}.
\]
Here and below, the inequality is applied separately to the real and imaginary parts. By Lemma~\ref{lem:wardgradient},
\[
\begin{aligned}
\left(
\Eds\left[
\left|
F_{{\bm x},{\bm y}}^{(E)}
-
\Eds F_{{\bm x},{\bm y}}^{(E)}
\right|^p
\right]
\right)^{\frac{1}{p}}
&\leqslant
\frac{C_q\sqrt{p}}{N^{\frac{q}{4}}\eta}
\left(
\Eds\left[
\left(
Y_{\bm x}^{(E)}Y_{\bm y}^{(E)}
\right)^{p/2}
\right]
\right)^{\frac{1}{p}}
\leqslant
\frac{C_q\sqrt{p}}{N^{\frac{q}{4}}\eta},
\end{aligned}
\]
where the last inequality follows from Cauchy--Schwarz and
Lemma~\ref{lem:wardmoments}, provided that
$p\leqslant cN^{\frac{q}{2}}\eta^2$.

Choosing $p$ proportional to $N^{\frac{q}{2}}\eta^2t^2$ and using Markov's inequality, we obtain, for every $0<t\leqslant1$,
\begin{equation}
\label{eq:fixxebound}
\mathds{P}\left(
\left|
F_{{\bm x},{\bm y}}^{(E)}(J)
-
\Eds [F_{{\bm x},{\bm y}}^{(E)}(J)]
\right|
\geqslant t
\right)
\leqslant
C\exp\left(-cN^{\frac{q}{2}}\eta^2t^2\right).
\end{equation}
If $N^{\frac{q}{2}}\eta^2t^2$ is bounded, this estimate follows by increasing the constant $C$.

By Lemma~\ref{lem:exptrick},
\[
\Eds[F_{{\bm x},{\bm y}}^{(E)}]
=
m_{N,\tau}(E+\i\eta)
\langle{\bm x},{\bm y}\rangle.
\]
Therefore, Proposition~\ref{pro:higherstab} and \eqref{eq:fixxebound} give
\[
\mathds{P}\left(
\left|
F_{{\bm x},{\bm y}}^{(E)}(J)
-
\Phi_R(E+\i\eta)\langle{\bm x},{\bm y}\rangle
\right|
\geqslant
\frac{C\alpha_{N,q}^{R+1}}{\eta^{2R+2}}+t
\right)
\leqslant
C\exp\left(-cN^{\frac{q}{2}}\eta^2t^2\right),
\]
where we used that $|\langle{\bm x},{\bm y}\rangle|\leqslant1$, and that for $q$ fixed and $\eta\geqslant N^{-\frac{1}{2}+\delta}$ we have $\alpha_{N,q}\eta^{-2}\leqslant CN^{-2\delta}$, so that the smallness assumption of Proposition~\ref{pro:higherstab} is satisfied for $N$ large enough.

To obtain uniformity in $E$, choose a net $\Qcal_N$ of $[-L,L]$ with mesh width $h_N=\eta^2t$, the letter $\Pcal$ having been used for the net in the proof of Theorem~\ref{theo:locallaw}.
The cardinality of this net is bounded by $|\Qcal_N|
\leqslant
C_L\left(1+\frac{1}{\eta^2t}\right).$
For every $E\in[-L,L]$, choose $E_0\in\Qcal_N$ such that
$|E-E_0|\leqslant h_N$. The deterministic resolvent bound gives
$\left|
F_{{\bm x},{\bm y}}^{(E)}-
F_{{\bm x},{\bm y}}^{(E_0)}
\right|
\leqslant
{|E-E_0|}{\eta^{-2}}
\leqslant t.
$
The same bound holds for
$|\Phi_R(E+\i\eta)-\Phi_R(E_0+\i\eta)|$, up to a constant $C_R$, since $\vert\Phi_R'\vert\leqslant C_R\eta^{-2}$ by \eqref{eq:gaussderivbound} and \eqref{eq:thetabounds}. A union bound over
$\Qcal_N$, followed by an adjustment of the constants, proves
\eqref{eq:isotropicmain}.
\end{proof}

\begin{proof}[Proof of Corollary~\ref{theo:delocinfinity}]
By spectral decomposition, for every $\eta>0$,
\begin{equation}
\label{eq:spectraldelocward}
\left|
\left\langle
{\bm u}_a^\tau,\psi_j^\tau
\right\rangle
\right|^2
\leqslant
\eta
\sum_{i=1}^{D^\tau}
\frac{
\left|
\left\langle
{\bm u}_a^\tau,\psi_i^\tau
\right\rangle
\right|^2\eta
}{
(\lambda_i^\tau-\lambda_j^\tau)^2+\eta^2
}
=
\eta
\Im\left\langle
{\bm u}_a^\tau,
G_\tau(\lambda_j^\tau+\i\eta)
{\bm u}_a^\tau
\right\rangle.
\end{equation}
Apply Theorem~\ref{theo:isolaw} with
\[
{\bm x}={\bm y}={\bm u}_a^\tau,
\qquad
t=1,
\qquad
\eta=u.
\]
Since $u\geqslant N^{-\frac{1}{2}+\delta}$ and $q$ is fixed,
\[
\frac{\alpha_{N,q}^{R+1}}{u^{2R+2}}
\leqslant
CN^{-2\delta(R+1)}
\leqslant C.
\]
Moreover, $\Im \Phi_R(E+\i u)$ is uniformly bounded, by \eqref{eq:gaussderivunif} and \eqref{eq:thetabounds}. Hence,
\[
\Pds \left(\max_{|E|\leqslant L}
\Im\left\langle
{\bm u}_a^\tau,
G_\tau(E+\i u)
{\bm u}_a^\tau
\right\rangle
> C\right)\leqslant 
C_L\left(1+\frac{1}{u^2}\right)
\exp\left(-cN^{\frac{q}{2}}u^2\right),
\]
We finish by taking a union bound over the deterministic basis, since $D^\tau=2^{\frac{N}{2}-1}$ and, for every fixed even $q\geqslant4$, $N^{\frac{q}{2}}u^2 \geqslant N^{\frac{q}{2}-1+2\delta}\geqslant N^{1+2\delta},$
the factor $D^\tau(1+u^{-2})$ can be absorbed into the exponential. Thus
\[
C_LD^\tau\left(1+\frac{1}{u^2}\right)
\exp\left(-cN^{\frac{q}{2}}u^2\right)
\leqslant
\exp\left(-c'N^{\frac{q}{2}}u^2\right)
\]
for all sufficiently large \(N\). Combining this estimate with
\eqref{eq:spectraldelocward} proves \eqref{eq:gooddeloc}.
\end{proof}

We thus finish by giving the proof of Proposition~\ref{pro:maincliffpro}.
\begin{proof}[Proof of Proposition~\ref{pro:maincliffpro}]
    We write the proof in the full Fock space. The proof is based on using the recent results of \cites{christiansen2024hilbert, visconti2026}. However, we first state the problem in a simpler way to apply these results. Consider $(e_S)_{S\subset\unn{1}{n}}$ the coordinate vectors of $\Fcal_n$ and write 
    \[
    \Fcal_n = \bigoplus_{m=0}^n \Fcal_{n,m} \quad\text{with}\quad 
    \Fcal_{n,m} = \text{span}\{e_S:\vert S\vert=m\}.
    \]
    Denote $\nu(j,S) = \vert \{s\in S:s<j\}\vert$ and the creation and annihilation operators defined on the coordinate basis by 
    \[
    a_je_S = \left\{
        \begin{array}{ll}
            (-1)^{\nu(j,S)}e_{S\setminus{j}}&\text{if }j\in S,\\
            0 & \text{otherwise}
        \end{array}
    \right.
    \quad \text{and}\quad 
    a_j^\ast e_S = \left\{
        \begin{array}{ll}
            (-1)^{\nu(j,S)}e_{S\cup{j}}&\text{if }j\not\in S,\\
            0 & \text{otherwise}.
        \end{array}
    \right.
    \]
    They satisfy the anticommutation relations
    \[
    a_ia_j+a_ja_i=0,
    \quad 
    a_i^\ast a_j^\ast+a_j^\ast a_i^\ast =0,
    \quad\text{and}\quad   
    a_ia_j^\ast+a_j^\ast a_i = \delta_{ij}\Id_D.
    \]
    In particular, we can write the Majorana fermions as 
    \begin{equation}\label{eq:majoranatoa}
    \gamma_{2j-1} = a_j+a_j^\ast \quad\text{and}\quad  \gamma_{2j} = \i (a_j^\ast -a_j).
    \end{equation}
    For $I = \{i_1<\dots<i_p\}$, we denote $a_I = a_{i_p}\dots a_{i_1}$. For $\psi\in\Fcal_{n,m}$, $I, J$ a $p$-subset, and $K$ a $m-p$ subset,  we define matrices through
    \[
    (X_{p,\psi})_{K,I} \coloneqq \langle e_K,a_I\psi\rangle,
    \quad R_{p,\psi} \coloneqq X_{p,\psi}^\ast X_{p,\psi} \quad\text{such that }\quad 
    (R_{p,\psi})_{I,J} = \langle a_I\psi,a_J\psi\rangle.
    \]
    In particular, we have the properties 
    \[
    R_{p,\psi} \geqslant 0,\quad \Tr R_{p,\psi} = \binom{m}{p}\Vert\psi\Vert^2\quad\text{and}\quad 
    \Vert R_{p,\psi}\Vert_{\mathrm{HS}}^2 = \sum_{\substack{\vert I\vert =p\\ \vert J\vert =p}}\vert \langle a_I\psi,a_J\psi\rangle\vert^2.
    \]
    Here and in the rest of this proof $\Vert\cdot\Vert_{\mathrm{HS}}$ denotes the \emph{unnormalized} Hilbert--Schmidt norm, as in \eqref{eq:visconti} below. This corresponds exactly to the reduced density matrix of \cite{visconti2026}*{Theorem 1} which proves that 
    \begin{equation}\label{eq:visconti}
    \Vert R_{p,\psi}\Vert_{\mathrm{HS}}\leqslant C_p m^{\frac{p}{2}}\Vert \psi\Vert^2.
    \end{equation}
    We now use the identity for a more general reduced matrix. For a given matrix $B$, we define 
    \[
    Q_{a,b}(B) = \sum_{\substack{\vert I\vert=a\\ \vert J\vert =b}}B_{I,J}a_I^\ast a_J.
    \]
    For $\psi \in \Fcal_{n,m}$ and $\phi\in \Fcal_{n,m+a-b}$ we have that $Q_{a,b}(B)\psi$ and $\phi$ lie in the same space, and by adjointness we have 
    \[
    \langle \phi,Q_{a,b}(B)\psi\rangle =\sum_{I,J}B_{I,J}\langle a_I\phi,a_J\psi\rangle.
    \]
    If we define the matrix 
    \[
    R^{(a,b)}_{\phi,\psi} = X^\ast_{a,\phi}X_{b,\psi} 
    \quad\text{such that }\quad 
    (R^{(a,b)}_{\phi,\psi})_{I,J} = \langle a_I \phi,a_J\psi\rangle,
    \]
    then by the Cauchy--Schwarz inequality, we have 
    \begin{equation}\label{eq:boundQab}
    \begin{aligned}
    \vert \langle \phi, Q_{a,b}(B)\psi\rangle\vert \leqslant \Vert B\Vert_{\mathrm{HS}}\Vert R^{(a,b)}_{\phi,\psi}\Vert_{\mathrm{HS}}
    =
    \Vert B\Vert_{\mathrm{HS}}\sqrt{\Tr(X_{a,\phi}X_{a,\phi}^\ast X_{b,\psi}X_{b,\psi}^\ast}) 
    &\leqslant \Vert B\Vert_{\mathrm{HS}}\sqrt{\Vert R_{a,\phi}\Vert_{\mathrm{HS}}\Vert R_{b,\psi}\Vert_{\mathrm{HS}}}
    \\& \leqslant C_{a,b} n^{\frac{a+b}{4}}\Vert B\Vert_{\mathrm{HS}}\Vert \phi\Vert \Vert \psi\Vert
    \end{aligned}
    \end{equation}
    where we used Visconti's inequality \eqref{eq:visconti} in the last step. We now need to bring back these estimates to the Majorana Fermions, denote $(b_1,\dots,b_{2n}) = (a_1,\dots,a_n,a_1^\ast,\dots,a_n^\ast)$ then for $(\theta_A)_{A}$ a family indexed by $q$-subsets,
    \[
    T_\theta \coloneqq \sum_{\substack{A\subset\unn{1}{N}\\\vert A\vert=q}}\theta_A\Gamma_A 
    =
    \sum_{j_1,\dots,j_q=1}^{2n} \vartheta_{j_1,\dots,j_q}b_{j_1}\dots b_{j_q}
    \]
    where we unfolded the product of Majorana fermions in terms of $a_i$'s and $a_i^\ast$'s through \eqref{eq:majoranatoa}. Since we have 
    \[
    \begin{bmatrix}
        \gamma_{2j-1}\\\gamma_{2j}
    \end{bmatrix}
    =
    U \begin{bmatrix}
        a_j \\ a_j^\ast
    \end{bmatrix}
    \quad\text{with}\quad 
    U = \begin{bmatrix}
        1 & 1\\ -\i & \i
    \end{bmatrix}
    \quad\text{which follows}\quad UU^\ast = 2\Id_2,
    \]
    we obtain that $\Vert \vartheta\Vert_2 \leqslant 2^{\frac{q}{2}}\Vert \theta\Vert_2.$ In the definition of $Q_{a,b}(B)$ we see that the creation and annihilation operators are ordered so that $a_I^\ast$ is on the left of $a_J$. To create this order, we can use the identity 
    \begin{equation}\label{eq:ordering}
    a_ja_k^\ast = -a_k^\ast a_j +\delta_{jk}\Id_D,
    \end{equation}
    where the first term only exchanges the two operators while the second term removes both. We call this second term a contraction. In particular, a term containing $\ell$ contractions has $q-2\ell$ operators left in its product. To see exactly the effect of a contraction, consider $\bm{\mu} = (\mu_1,\dots,\mu_{q-2})$ the remaining indices after one contraction, if we consider the term 
    \[
    \sum_{j,k=1}^n \vartheta_{j,n+k,\bm{\mu}}a_ja_k^\ast b_{\mu_1}\dots b_{{\mu}_{q-2}}
    =
    -\sum_{j,k=1}^n \vartheta_{j,n+k,\bm{\mu}}a_k^\ast a_j b_{\mu_1}\dots b_{{\mu}_{q-2}}
    +
    \sum_{j=1}^n \vartheta_{j,n+j,\bm{\mu}}b_{\mu_1}\dots b_{{\mu}_{q-2}}.
    \]
    Thus we can consider the contraction operator $\Ccal$ defined by 
    \[
    (\Ccal \vartheta)_{\bm{\mu}} \coloneqq \sum_{j=1}^n \vartheta_{j,n+j,\bm{\mu}}
    \quad\text{with}\quad 
    \Vert \Ccal\vartheta\Vert_2 \leqslant \sqrt{n}\Vert \vartheta\Vert_2
    \]
    by Cauchy--Schwarz. And for a fixed pattern of $\ell$ contractions, we have $\Vert \Ccal^\ell \vartheta\Vert_2\leqslant n^{\frac{\ell}{2}}\Vert \vartheta\Vert_2$. After $\ell$ contractions have been performed, we put the remaining $q-2\ell$ operators in $a_I^\ast a_J$ order. Thus if $a$ operators are creation operators and $b$ are annihilation operators we have the relation $a+b=q-2\ell$. Finally, this means that we can write 
    \[
    T_\theta = \sum_{\ell=0}^{\frac{q}{2}} \sum_{a+b = q-2\ell} Q_{a,b}(B_{a,b}^{(\ell)})
    \]
    where the matrices $B_{a,b}^{(\ell)}$ collect all the ordering terms coming from the $\delta_{jk}$ term  in \eqref{eq:ordering}. Note that the signs, permutations and sums over patterns in these matrices only depend on $q$ which is independent of $N$ in this proof. Thus, from our bound on the contraction operators we get that 
    \begin{equation}\label{eq:boundBab}
    \left(  
        \sum_{a+b=q-2\ell} \Vert B_{a,b}^{(\ell)}\Vert_{\mathrm{HS}}^2
    \right)^{\frac{1}{2}}  
    \leqslant C_q n^{\frac{\ell}{2}}\Vert \vartheta\Vert_2
    \end{equation}
    for some $C_q>0$. Although the estimate \eqref{eq:boundQab} was derived for vectors with fixed particle number, it extends to arbitrary vectors in $\Fcal_n$ at the cost of changing the constant $C_q$. Indeed, write   
    \[
        v = \sum_{m=0}^n v_m,\quad w=\sum_{m=0}^n w_m
        \quad\text{with}\quad v_m,w_m\in\Fcal_{n,m}.
    \]
    Each normally ordered term $Q_{a,b}(B)$ maps $\Fcal_{n,m}$ into $\Fcal_{n,m+a-b}$ and since $a+b=q-2\ell$, the operator $T_\theta$ connects only particle number sectors whose indices differ by $d\in\{-q,-q+2,\dots,q-2,q\}$ and by Cauchy--Schwarz we have 
    \[
    \sum_{m=0}^n \Vert v_{m+d}\Vert \Vert w_m\Vert \leqslant \Vert v\Vert\Vert w\Vert.
    \]
    Summing over the $q+1$ possible values of $d$ only changes the constant $C_q$. We thus have the operator-norm estimate on the whole Fock space $\Fcal_n$, 
    \[
    \Vert T_\theta\Vert  
    \leqslant  
    \sum_{\ell=0}^{\frac{q}{2}}
    \sum_{a+b=q-2\ell} C_{a,b}n^{\frac{a+b}{4}}\Vert B_{a,b}^{(\ell)}\Vert_{\mathrm{HS}}
    \leqslant C_q \sum_{\ell=0}^{\frac{q}{2}} n^{\frac{q-2\ell}{4}} 
    \left(  
        \sum_{a+b=q-2\ell} \Vert B_{a,b}^{(\ell)}\Vert_{\mathrm{HS}}^2
    \right)^{\frac{1}{2}}  
    \leqslant C_q n^{\frac{q}{4}}\Vert \vartheta\Vert_2
    \leqslant C_q n^{\frac{q}{4}}\Vert \theta\Vert_2,
    \]
    where we change the value of $C_q$ in each inequality to incorporate the combinatorial terms only depending on $q$. This inequality thus directly gives that
    \[
    \left( 
        \sum_{\substack{A\subset\unn{1}{N}\\\vert A\vert=q}} \vert \langle v,\Gamma_A w\rangle\vert^2
    \right)^{\frac{1}{2}} 
    = \sup_{\Vert \theta\Vert_2=1} \left\vert 
        \sum_{\vert A\vert =q}\theta_A \langle v,\Gamma_A w\rangle
    \right\vert 
    =
    \sup_{\Vert \theta\Vert_2 =1} \vert \langle v,T_\theta w\rangle\vert \leqslant \sup_{\Vert \theta\Vert_2=1}
    \Vert T_\theta\Vert\Vert v\Vert \Vert w\Vert \leqslant C_q n^{\frac{q}{4}}\Vert v\Vert \Vert w\Vert
    \]
    and after renormalization, updating $C_q$ a final time,
    \[
    \frac{1}{M}\sum_{\substack{A\subset\unn{1}{N}\\\vert A\vert=q}} \vert \langle v,\Gamma_A w\rangle\vert^2  
    \leqslant C_qN^{-q}N^{\frac{q}{2}}\Vert v\Vert^2\Vert w\Vert^2 = C_q N^{-\frac{q}{2}}\Vert v\Vert^2\Vert w \Vert^2
    \]
    giving the final result.
\end{proof}

\appendix
\appendixswitch

 \section{Additional technical results}
 \label{app:techres}

 In this appendix we collect the proofs of the elementary properties of the Stieltjes transform $m$ of the standard Gaussian law stated in Lemma~\ref{lem:gaussprofile}.

\begin{proof}[Proof of Lemma~\ref{lem:gaussprofile}]
    Differentiating under the integral sign gives the first identity in \eqref{eq:mkrepr}, and integrating by parts $k$ times gives the second one; in particular \eqref{eq:gaussderivbound} follows from the first identity. Since $\rho_G'(x)=-x\rho_G(x)$, an integration by parts yields
    \[
    m'(z)=\int_\Rbb\frac{\rho_G(x)}{(x-z)^2}\d x=\int_\Rbb \frac{\rho_G'(x)}{x-z}\d x=-\int_\Rbb \frac{x\rho_G(x)}{x-z}\d x=-1-zm(z),
    \]
    which is the first identity in \eqref{eq:gaussode}. Differentiating it $k$ times, and noticing that only the terms $i=0$ and $i=1$ survive in the Leibniz expansion of $(zm)^{(k)}$, gives the second one. Applying it with $k=j+1$ gives $L[-\frac{1}{j+1}m^{(j+1)}]=m^{(j)}$, and $m^{(j+1)}(z)\to0$ as $\Im z\to\infty$ by \eqref{eq:gaussderivbound}; uniqueness holds because the solutions of the homogeneous equation $L[h]=0$ are the multiples of $\e^{-\frac{z^2}{2}}$, which are unbounded along vertical lines unless they vanish identically.

    Finally, $\rho_G$ extends to the entire function $\rho_G(\zeta)=(2\pi)^{-\frac{1}{2}}\e^{-\frac{\zeta^2}{2}}$, which satisfies $\vert \rho_G(t-3\i)\vert = \e^{\frac{9}{2}}\rho_G(t)$ for $t\in\Rbb$. For $z\in\Cbb_+$ the map $\zeta\mapsto \rho_G(\zeta)(\zeta-z)^{-k-1}$ is holomorphic on the strip $\{-3\leqslant \Im\zeta\leqslant0\}$ and decays like a Gaussian as $\vert \Re\zeta\vert \to\infty$ there, so that we may shift the contour of integration in \eqref{eq:mkrepr} and obtain
    \[
    m^{(k)}(z)=k!\int_\Rbb \frac{\rho_G(t-3\i)}{(t-3\i-z)^{k+1}}\d t.
    \]
    The right-hand side is holomorphic in $\{\Im z>-3\}$, which gives the desired extension. Moreover, if $0<\Im z\leqslant2$ then $\vert t-3\i-z\vert \geqslant 3+\Im z\geqslant3$, and therefore $\vert m^{(k)}(z)\vert \leqslant k!\,3^{-k-1}\e^{\frac{9}{2}}$, which is \eqref{eq:gaussderivunif}.
\end{proof}

%\noindent\textbf{Data availability statement:} No data were generated or analyzed during this work.\\
%\textbf{Conflict of interest statement:} The authors declare that they have no conflict of interest.
\bibliographystyle{abbrv}
\bibliography{biblio}

\begin{bibdiv}
\begin{biblist}

\bib{anschuetz2025}{article}{
      author={Anschuetz, E.~R.},
      author={Chen, C.-F.},
      author={Kiani, B.~T.},
      author={King, R.},
       title={Strongly interacting fermions are nontrivial yet nonglassy},
        date={2025Jul},
     journal={Phys. Rev. Lett.},
      volume={135},
       pages={030602},
         url={https://link.aps.org/doi/10.1103/cbqf-d24r},
}

\bib{berkooz2020susy}{article}{
      author={Berkooz, M.},
      author={Brukner, N.},
      author={Narovlansky, V.},
      author={Raz, A.},
       title={The double scaled limit of super-symmetric {SYK} models},
        date={2020},
     journal={J. High Energy Phys.},
      volume={2020},
      number={12},
       pages={110},
}

\bib{berkooz2019}{article}{
      author={Berkooz, M.},
      author={Isachenkov, M.},
      author={Narovlansky, V.},
      author={Torrents, G.},
       title={Towards a full solution of the large {$N$} double-scaled {SYK}
  model},
        date={2019},
     journal={J. High Energy Phys.},
      volume={2019},
      number={3},
       pages={079},
}

\bib{basu2026sharp}{article}{
      author={Basu, A.},
      author={Kothari, P.~K.},
      author={Midha, S.},
       title={Sharp bounds on ground state energy of the {SYK} model},
        date={2026},
     journal={arXiv preprint arXiv:2607.27185},
}

\bib{bozejko97}{article}{
      author={Bo\.zejko, M.},
      author={K\"ummerer, B.},
      author={Speicher, R.},
       title={{$q$}-{G}aussian processes: non-commutative and classical
  aspects},
        date={1997},
        ISSN={0010-3616,1432-0916},
     journal={Comm. Math. Phys.},
      volume={185},
      number={1},
       pages={129\ndash 154},
}

\bib{berkoozmamroud2025}{article}{
      author={Berkooz, M.},
      author={Mamroud, O.},
       title={A cordial introduction to double scaled {SYK}},
        date={2025feb},
     journal={Reports on Progress in Physics},
      volume={88},
      number={3},
       pages={036001},
         url={https://doi.org/10.1088/1361-6633/ada889},
}

\bib{brai1980replica}{article}{
      author={Bray, A.~J.},
      author={Moore, M.~A.},
       title={Replica theory of quantum spin glasses},
        date={1980Aug},
     journal={J. Phys. C},
      volume={13},
      number={24},
       pages={L655},
         url={https://doi.org/10.1088/0022-3719/13/24/005},
}

\bib{berkooz2021complex}{article}{
      author={Berkooz, M.},
      author={Narovlansky, V.},
      author={Raj, H.},
       title={Complex {S}achdev-{Y}e-{K}itaev model in the double scaling
  limit},
        date={2021},
     journal={J. High Energy Phys.},
      volume={2021},
      number={2},
       pages={113},
}

\bib{bozejko91}{article}{
      author={Bo\.zejko, M.},
      author={Speicher, R.},
       title={An example of a generalized {B}rownian motion},
        date={1991},
        ISSN={0010-3616,1432-0916},
     journal={Comm. Math. Phys.},
      volume={137},
      number={3},
       pages={519\ndash 531},
}

\bib{christiansen2024hilbert}{article}{
      author={Christiansen, M.~R.},
       title={Hilbert-{S}chmidt estimates for fermionic 2-body operators},
        date={2024},
        ISSN={0010-3616,1432-0916},
     journal={Comm. Math. Phys.},
      volume={405},
      number={1},
       pages={Paper No. 18, 9},
}

\bib{erdos2014phase}{article}{
      author={Erd\H{o}s, L.},
      author={Schr\"{o}der, D.},
       title={Phase transition in the density of states of quantum spin
  glasses},
        date={2014},
        ISSN={1385-0172,1572-9656},
     journal={Math. Phys. Anal. Geom.},
      volume={17},
      number={3-4},
       pages={441\ndash 464},
}

\bib{feng2019spectrum}{article}{
      author={Feng, R.},
      author={Tian, G.},
      author={Wei, D.},
       title={Spectrum of {SYK} model},
        date={2019},
        ISSN={2096-6075,2524-7182},
     journal={Peking Math. J.},
      volume={2},
      number={1},
       pages={41\ndash 70},
}

\bib{feng2020spectrum}{article}{
      author={Feng, R.},
      author={Tian, G.},
      author={Wei, D.},
       title={Spectrum of {SYK} model {III}: {L}arge deviations and
  concentration of measures},
        date={2020},
        ISSN={2010-3263,2010-3271},
     journal={Random Matrices Theory Appl.},
      volume={9},
      number={2},
       pages={2050001, 24},
}

\bib{feng2021spectrum}{article}{
      author={Feng, R.},
      author={Tian, G.},
      author={Wei, D.},
       title={Spectrum of {SYK} model {II}: central limit theorem},
        date={2021},
        ISSN={2010-3263,2010-3271},
     journal={Random Matrices Theory Appl.},
      volume={10},
      number={4},
       pages={Paper No. 2150037, 30},
}

\bib{garciagarcia2018}{article}{
      author={Garc\'{i}a-Garc\'{i}a, A.~M.},
      author={Jia, Y.},
      author={Verbaarschot, J. J.~M.},
       title={Exact moments of the {S}achdev-{Y}e-{K}itaev model up to order
  {$1/N^2$}},
        date={2018},
     journal={J. High Energy Phys.},
      volume={2018},
      number={4},
       pages={146},
}

\bib{garciagarcia2016}{article}{
      author={Garc\'{i}a-Garc\'{i}a, A.~M.},
      author={Verbaarschot, J. J.~M.},
       title={Spectral and thermodynamic properties of the
  {S}achdev-{Y}e-{K}itaev model},
        date={2016},
     journal={Phys. Rev. D},
      volume={94},
       pages={126010},
}

\bib{garciagarcia2017}{article}{
      author={Garc\'{i}a-Garc\'{i}a, A.~M.},
      author={Verbaarschot, J. J.~M.},
       title={Analytical spectral density of the {S}achdev-{Y}e-{K}itaev model
  at finite {$N$}},
        date={2017},
     journal={Phys. Rev. D},
      volume={96},
       pages={066012},
}

\bib{gamarnik2026free}{article}{
      author={Gamarnik, D.},
      author={Pernice, F.},
      author={Schmidhuber, A.},
      author={Zlokapa, A.},
       title={The free energy limit of the {SYK} model at high temperature},
        date={2026},
     journal={arXiv preprint arXiv:2605.02768},
}

\bib{he2026spectral}{article}{
      author={He, Y.},
       title={The spectral edge of the quartic {SYK} model},
        date={2026},
     journal={arXiv preprint arXiv:2607.18998},
}

\bib{haque2019}{article}{
      author={Haque, M.},
      author={McClarty, P.~A.},
       title={Eigenstate thermalization scaling in {M}ajorana clusters: From
  chaotic to integrable {S}achdev-{Y}e-{K}itaev models},
        date={2019Sep},
     journal={Phys. Rev. B},
      volume={100},
       pages={115122},
         url={https://link.aps.org/doi/10.1103/PhysRevB.100.115122},
}

\bib{hastings2022}{inproceedings}{
      author={Hastings, M.~B.},
      author={O'Donnell, R.},
       title={Optimizing strongly interacting fermionic {H}amiltonians},
        date={2022},
   booktitle={S{TOC} '22---{P}roceedings of the 54th {A}nnual {ACM} {SIGACT}
  {S}ymposium on {T}heory of {C}omputing},
   publisher={ACM, New York},
       pages={776\ndash 789},
}

\bib{jia2018}{article}{
      author={Jia, Y.},
      author={Verbaarschot, J. J.~M.},
       title={Large {$N$} expansion of the moments and free energy of
  {S}achdev-{Y}e-{K}itaev model, and the enumeration of intersection graphs},
        date={2018},
     journal={J. High Energy Phys.},
      volume={2018},
      number={11},
       pages={031},
}

\bib{jia2020}{article}{
      author={Jia, Y.},
      author={Verbaarschot, J. J.~M.},
       title={Spectral fluctuations in the {S}achdev-{Y}e-{K}itaev model},
        date={2020},
     journal={J. High Energy Phys.},
      volume={2020},
      number={7},
       pages={193},
}

\bib{Kitaev1}{misc}{
      author={Kitaev, A.},
       title={A simple model of quantum holography, part {I}},
        date={2015},
         url={https://online.kitp.ucsb.edu/online/entangled15/kitaev/},
        note={Talk at the Kavli Institute for Theoretical Physics (KITP)},
}

\bib{Kitaev2}{misc}{
      author={Kitaev, A.},
       title={A simple model of quantum holography, part {II}},
        date={2015},
         url={https://online.kitp.ucsb.edu/online/entangled15/kitaev/},
        note={Talk at the Kavli Institute for Theoretical Physics (KITP)},
}

\bib{keating2015}{article}{
      author={Keating, J.~P.},
      author={Linden, N.},
      author={Wells, H.~J.},
       title={Random matrices and quantum spin chains},
        date={2015},
        ISSN={1024-2953},
     journal={Markov Process. Related Fields},
      volume={21},
      number={3},
       pages={537\ndash 555},
}

\bib{Kanazawa2017}{article}{
      author={Kanazawa, T.},
      author={Wettig, T.},
       title={{Complete random matrix classification of {SYK} models with
  $\mathcal{N}= 0, 1$ and 2 supersymmetry}},
        date={2017Sep},
        ISSN={1029-8479},
     journal={J. High Energ. Phys.},
      volume={2017},
      number={9},
       pages={50},
         url={https://doi.org/10.1007/JHEP09(2017)050},
}

\bib{maldacena2016remarks}{article}{
      author={Maldacena, J.},
      author={Stanford, D.},
       title={Remarks on the {S}achdev--{Y}e--{K}itaev model},
        date={2016Nov},
     journal={Phys. Rev. D},
      volume={94},
       pages={106002},
         url={https://link.aps.org/doi/10.1103/PhysRevD.94.106002},
}

\bib{polchinski2016spectrum}{article}{
      author={Polchinski, J.},
      author={Rosenhaus, V.},
       title={The spectrum in the {S}achdev--{Y}e--{K}itaev model},
        date={2016},
     journal={J. High Energy Phys.},
      volume={2016},
      number={4},
       pages={1\ndash 25},
}

\bib{Sonner2017}{article}{
      author={Sonner, J.},
      author={Vielma, M.},
       title={Eigenstate thermalization in the {S}achdev-{Y}e-{K}itaev model},
        date={2017Nov},
        ISSN={1029-8479},
     journal={J. High Energy Phys.},
      volume={2017},
      number={11},
       pages={149},
         url={https://doi.org/10.1007/JHEP11(2017)149},
}

\bib{sun2020periodic}{article}{
      author={Sun, F.},
      author={Ye, J.},
       title={{Periodic table of the ordinary and supersymmetric
  Sachdev-Ye-Kitaev models}},
        date={2020},
     journal={Phys. Rev. Lett.},
      volume={124},
      number={24},
       pages={244101},
}

\bib{sachdevye}{article}{
      author={Sachdev, S.},
      author={Ye, J.},
       title={Gapless spin-fluid ground state in a random quantum {H}eisenberg
  magnet},
        date={1993May},
     journal={Phys. Rev. Lett.},
      volume={70},
       pages={3339\ndash 3342},
         url={https://link.aps.org/doi/10.1103/PhysRevLett.70.3339},
}

\bib{visconti2026}{article}{
      author={Visconti, F. L.~A.},
       title={{H}ilbert--{S}chmidt norm estimates for fermionic reduced density
  matrices},
        date={2026Mar},
        ISSN={1424-0661},
     journal={Ann. Henri Poincar\'e},
         url={https://doi.org/10.1007/s00023-026-01686-z},
}

\bib{you2017sachdev}{article}{
      author={You, Y.-Z.},
      author={Ludwig, A. W.~W.},
      author={Xu, C.},
       title={{S}achdev-{Y}e-{K}itaev model and thermalization on the boundary
  of many-body localized fermionic symmetry-protected topological states},
        date={2017Mar},
     journal={Phys. Rev. B},
      volume={95},
       pages={115150},
         url={https://link.aps.org/doi/10.1103/PhysRevB.95.115150},
}

\end{biblist}
\end{bibdiv}
\end{document}